\documentclass[12pt]{opt2019} 

\usepackage{multirow}
\usepackage{algorithm,algorithmicx}
\usepackage{framed}
\usepackage{enumerate}

\newtheorem{assump}{Assumption}
\newtheorem{defn}{Definition}[section]

\usepackage{tikz}
\usetikzlibrary{calc,trees,positioning,arrows,chains,shapes.geometric,%
	decorations.pathreplacing,decorations.pathmorphing,shapes,%
	matrix,shapes.symbols}

\tikzstyle{startstop} = [rectangle, draw, rounded corners, align=center, minimum width=3cm, minimum height=1cm,text centered]
\tikzstyle{decision} = [diamond, draw, fill=blue!20, 
text width=4.5em, text badly centered, node distance=3cm, inner sep=0pt]
\tikzstyle{block} = [rectangle, draw, fill=blue!10, align=center, rounded corners, minimum width=3cm, minimum height=1cm]
\tikzstyle{blockcast} = [rectangle, draw, fill=red!10, align=center, rounded corners, minimum width=3cm, minimum height=0.45cm]
\tikzstyle{line} = [draw, -latex']
\tikzstyle{cloud} = [draw, ellipse,fill=red!20, node distance=3cm,
minimum height=2em]

\title[Byzantine Resilient Non-Convex SVRG with Distributed Batch Gradient Computations]{Byzantine Resilient Non-Convex SVRG with Distributed Batch Gradient Computations}

\optauthor{\Name{Prashant Khanduri} \Email{pkhandur@syr.edu}\\
	\Name{Saikiran Bulusu} \Email{sabulusu@syr.edu}\\
	\Name{Pranay Sharma} \Email{psharm04@syr.edu}\\
	\Name{Pramod K. Varshney} \Email{varshney@syr.edu}\\
	\addr Syracuse University, Syracuse, NY 13244}

\begin{document}
	
	\maketitle
	
	\begin{abstract}%
		In this work, we consider the distributed stochastic optimization problem of minimizing a non-convex function $f(x) = \mathbb{E}_{\xi \sim \mathcal{D}} f(x; \xi)$ in an adversarial setting, where the individual functions $f(x; \xi)$ can also be potentially non-convex. We assume that at most $\alpha$-fraction of a total of $K$ nodes can be Byzantines. We propose a robust stochastic variance-reduced gradient (SVRG) like algorithm for the problem, where the batch gradients are computed at the worker nodes (WNs) and the stochastic gradients are computed at the server node (SN). For the non-convex optimization problem, we show that we need $\tilde{O}\left( \frac{1}{\epsilon^{5/3} K^{2/3}}  + \frac{\alpha^{4/3}}{\epsilon^{5/3}} \right)$ gradient computations on average at each node (SN and WNs) to reach an $\epsilon$-stationary point. The proposed algorithm guarantees convergence via the design of a novel Byzantine filtering rule which is independent of the problem dimension. Importantly, we capture the effect of the fraction of Byzantine nodes $\alpha$ present in the network on the convergence performance of the algorithm. 
	\end{abstract}
	
	
	\section{Introduction}
	In the current machine learning landscape, with the data sizes growing exponentially, distributed learning has become an important paradigm. In a distributed learning setup, multiple distributed nodes along with a server node perform the learning task at hand. In distributed settings, the computational load on the server node is relieved as heavy computations are usually distributed across the multiple nodes in the network, see \citet{Zinkevich_NIPS_2010,Recht_NIPS_2011,Dekel_JMLR_2012,Ho_NIPS_2013}. With multiple nodes in the network, the problem of robust learning becomes important as some of the nodes in the network can potentially be Byzantine. These Byzantine nodes can adversely affect the convergence performance of the algorithm. Therefore, it is important to design algorithms which are robust to Byzantine actions and at the same time provide sufficient convergence guarantees. In this work, we propose one such algorithm for distributed learning, when a maximum $\alpha \in [0 , \frac{1}{2})$ fraction of nodes in the network are Byzantines. 
	
	Many works in the past have looked into the problem of learning in the presence of Byzantines, e.g., in \citet{Alistarh_NIPS_2018, Su_2018_securing, Recht_NIPS_2011, Yin_PMLR_2018,Yin_PMLR_2019,Blanchard_NIPS_2017, Xie_2018_phocas,Xie_2018_zeno}. However, only a handful have considered the problems when the objective functions are non-convex, partially, because of the difficulty in dealing with the non-convex objective functions, see \citet{Yin_PMLR_2018,Yin_PMLR_2019,Blanchard_NIPS_2017,Xie_2018_phocas,Xie_2018_zeno}. In this work, we deal with a non-convex optimization problem in the presence of Byzantines. To the best of our knowledge, this is the first work which provides guarantees for an SVRG-like algorithm (Sections \ref{Sec: Model} and \ref{Sec: Algo}) for distributed learning in the presence of Byzantines with non-convex objective functions. 
	
	\paragraph{Related Work:} In \citet{Su_2018_securing}, a robust gradient descent (GD) based algorithm was proposed in the presence of Byzantine nodes for learning with strongly convex objective functions. Also, in \citet{Yin_PMLR_2018,Yin_PMLR_2019}, the authors proposed Byzantine resilient GD algorithms for non-convex objective functions. Note that GD based algorithms require computations of the gradients of complete batches, therefore, in \citet{Blanchard_NIPS_2017,Xie_2018_zeno,Xie_2018_phocas,Li_AAAI_2019,Alistarh_NIPS_2018}, the authors proposed stochastic gradient descent (SGD) based robust learning algorithms in the presence of Byzantine nodes. In particular, \citet{Li_AAAI_2019, Alistarh_NIPS_2018}, considered the objective functions to be strongly convex and convex, respectively. On the other hand, in \citet{Blanchard_NIPS_2017} and \citet{Xie_2018_zeno,Xie_2018_phocas}, the authors proposed robust algorithms in the presence of Byzantine nodes for non-convex objective functions. 
	
	In this work, we consider a non-convex learning problem in the presence of at most $\alpha$-fraction Byzantine nodes. The proposed algorithm in this work is based on the stochastic variance reduced gradient (SVRG) and stochastically controlled stochastic gradient (SCSG) frameworks proposed in \citet{Reddy_PMLR_2016} and \citet{Lei_Jordan_SCSG}, respectively, for non-convex objective functions. SVRG and SCSG reduce the variance of stochastic gradients in SGD by introducing inner and outer loops and by computing a batch gradient at the beginning of every outer loop. SCSG based algorithms originally proposed in \citet{Lei_Jordan_PMLR_2017} have been shown to improve the performance of SVRG by selecting the number of inner iterations in a random fashion, whereas SVRG chooses the inner loop size in a deterministic fashion. Our work, extends the framework of non-convex SCSG to distributed setting, even in the presence of Byzantine nodes. Below, we list our contributions:
	\paragraph{Contributions:} 
	\begin{itemize}
		\item We propose a novel algorithm for distributed non-convex optimization in the presence of Byzantine nodes. The proposed algorithm is a variant of the SVRG algorithm (Section \ref{Sec: Algo}), where the computationally demanding batch gradients are computed at the worker nodes (WNs) and the stochastic gradients are computed at the server node (SN). 
		\item We provide convergence guarantees for the proposed algorithm as a function of $\alpha \in [0,\frac{1}{2})$, which is the upper bound on the fraction of Byzantine nodes present in the network. Importantly, we show that as the number of Byzantines reduces to zero ($\alpha = 0$), the algorithm improves the best known convergence rates for distributed non-convex optimization (Section \ref{Sec: Conv}) \citet{Hao_ICML_2019,Jiang_NIPS_2018}. 
		\item We propose a novel Byzantine filtering rule which is independent of the problem dimension. Specifically, the aggregation rule proposed in this work does not perform coordinate-wise operations and thereby the convergence rates are independent of the problem dimension. Therefore, the proposed algorithm is suitable for high-dimensional learning problems. 
	\end{itemize}

	\section{Model and Assumptions}
	\label{Sec: Model}
	We consider a model similar to the one in \citet{Alistarh_NIPS_2018}. We assume that there are $K$ WNs and a SN in the network. Specifically, we want to solve the following problem in a distributed fashion:
	\begin{align*}
		\min_{x \in \mathbb{R}^d} f(x) = \mathbb{E}_{\xi \sim \mathcal{D}} f( x; \xi)
	\end{align*}
	where the functions $f(\cdot~\! ; \xi): \mathbb{R}^d \rightarrow \mathbb{R}$, for $\xi$ chosen uniformly randomly from distribution $\mathcal{D}$, and  $f$, can potentially be non-convex. All the nodes including SN have access to the stochastic functions from distribution $\mathcal{D}$. Of all the WNs, we assume that at most $\alpha$-fraction are Byzantines with  $\alpha \in \left[ 0 , 1/2 \right)$.
	
	The set of honest nodes is denoted by $\mathcal{G}$. For each honest node $k \in \mathcal{G}$, the following assumptions are made.
	
	\begin{assump}[Gradient Lipschitz continuity] 
		\label{Ass: LipCont}
		All the functions $f(\cdot ~\! ; \xi)$ for any $\xi \sim \mathcal{D}$ and $f$ are assumed to be $L$-smooth, i.e., we have $f(y) \leq f(x) + \nabla f(x)^T(y - x) + \frac{L}{2} \| y - x \|^2_2,$ with $L > 0$.
	\end{assump}
	\begin{assump}[Bounded Variance]
		\label{Ass: BoundedGradVar}
		For any $\xi \sim \mathcal{D}$ we have $\| \nabla f(x ; \xi)  -  \nabla f(x)\| \leq \mathcal{V}$.
	\end{assump}
	\begin{remark}
		Assumption \ref{Ass: BoundedGradVar} is also required in \citet{Alistarh_NIPS_2018} to design the Byzantine filtering strategy. Moreover, for $\alpha = 0$, the above assumption can be relaxed to $\mathbb{E}\| \nabla f(x ; \xi)  -  \nabla f(x)\| \leq \mathcal{V}$, which is a standard assumption in stochastic non-convex optimization literature. 
	\end{remark}
	For non-convex problems, it is generally not feasible to measure the suboptimality of the function value, therefore, usually the convergence of non-convex problems is measured in terms of expected gradient norm square, $\mathbb{E} \|f(x) \|^2$. Below we define an $\epsilon$-stationary point for a non-convex optimization problem. 
	\begin{defn}[$\epsilon$-Stationary Point]
		A point $x$ is called $\epsilon$-stationary if $\| \nabla f(x) \|^2 \leq \epsilon$. Moreover, a stochastic algorithm is said to achieve $\epsilon$-stationarity in $t$ iterations if $\mathbb{E}[\| \nabla f(x_t) \|^2] \leq \epsilon$, where the expectation is over the stochasticity of the algorithm until time instant $t$.
	\end{defn}
	Next, we discuss the algorithm. 
	\section{Algorithm}
	\label{Sec: Algo}
	Now, we discuss the steps of Algorithm \ref{alg1}. As mentioned earlier, we consider a SVRG-like algorithm, which can be thought of as a distributed version of SCSG, where the WNs compute the batch gradients and share the computed batch gradients with the SN. Note that the algorithm is similar to SVRG except the fact that the algorithm uses a geometric random variable to decide the number of inner iterations. 
	
	The algorithm consists of a total of $T$ epochs. At the start of each epoch $t = 1, 2, \ldots, T$, the SN broadcasts the point $x_0^t$ to the WNs. The WNs are then expected to compute their batch gradients at $x_0^t$ and forward them to the SN. However, a Byzantine node forwards an arbitrary vector to the SN. The honest WNs, on the other hand, compute and forward their batch gradients, $\mu_{t}^{(k)}$. The set $\mathcal{S}_t^{(k)}$ for $t = 1,2, \ldots, T$ and $k \in [K]$ consists of $B$ i.i.d. sample functions $f(x_0^t; \xi_{t,i}^{(k)})$ with $\xi_{t,i}^{(k)} \sim \mathcal{D}$ and $\mathcal{S}_t^{(k)} = (\xi_{t,i}^{(k)})_{i=1}^B$. Note that for simplicity, we assume the batch sizes $|\mathcal{S}_t^{(k)}| = B$, for all $t = 1,2, \ldots, T$ and $k \in \mathcal{G}$. Furthermore, a generalization with variable batch sizes at different nodes is straightforward. To summarize, the vectors sent by node $k$ at the $t^{\text{th}}$ epoch to the SN is:
	\begin{align}\label{eq: Grad}
		\mu_t^{(k)} = \begin{cases} 
			\frac{1}{B}\sum_{i=1}^B \nabla f(x_0^t ; \xi_{t,i}^{(k)}) & \text{if}~k \in \mathcal{G}\\
			\ast & \text{if}~k \notin \mathcal{G}
		\end{cases}
	\end{align}
	where $\ast$ indicates an arbitrary vector sent by the Byzantine node. 
	
	After receiving the batch gradients from the WNs, the SN performs a {Byzantine Filtering Step}. In this step, the SN computes its estimate of the good set $\mathcal{G}_t$. The server node then aggregates the gradients received from the WNs in the set $\mathcal{G}_t$ and forms an estimate of the batch gradient, $\mu_t$. This is followed by the SVRG-like inner loop ($n = 1, 2,\ldots, N_t$), where the inner loop size $N_t$ is chosen randomly using a geometric random variable with parameter $\frac{B}{B + 1}$, i.e., $N_t \sim \text{Geom}\left(\frac{B}{B + 1}\right)$. The SN then performs the SVRG update step using the estimated batch gradient, $\mu_t$, and stochastic gradients computed at $x_{n - 1}^{t}$ and $x_0^t$. Next, we discuss the proposed Byzantine filtering strategy.
	
	\paragraph{Byzantine Filtering Step:} To design the Byzantine filtering rule, the SN uses the computed batch gradients from all the worker nodes, $(\mu_t^{(k)})_{k \in [K]}$ . Note that some of the received batch gradients may be arbitrary sent by the Byzantines (see \eqref{eq: Grad}).  Below we define the Byzantine filtering rule and then motivate its construction. 
	
	First, the SN computes the vector median of the received vectors $(\mu_t^{(k)})_{k \in [K]}$ using $\mathfrak{T}_\mu$ as defined in Algorithm \ref{alg1}. Then the SN filters the nodes it {\em believes} to be Byzantines and constructs $\mathcal{G}_t$ using the rule: $\mathcal{G}_t = \{k \in [K] : \| \mu_t^{(k)} - \mu_t^{\text{med}} \| \leq 2 \mathfrak{T}_\mu \}.$
	\begin{algorithm}[t] 
		\KwIn{$\tilde{x}_0 \in \mathbb{R}^d$, step sizes $(\eta_t)_{t = 1}^T$, batch size $B$, Variance Bound $\mathcal{V}$ (Assumption \ref{Ass: BoundedGradVar}), $\mathfrak{T}_{\mu} = 2 \mathcal{V} \sqrt{\frac{C}{B}}$ (Lemma \ref{lem: T_mu}) where $C = 2 \log \left( \frac{2K}{\delta} \right)$ with $\delta \in (0,1)$ (Theorem \ref{Thm:Convergence_BSVRG}).}
		\For{t = 1,2, \ldots, T}{ $x_0^t \leftarrow \tilde{x}_{t-1}$ \qquad{\bf  $\rightarrow$ Push to WNs};
			
			\For{k = 1,2,\ldots, K }{\vspace{- 0.1 in}\begin{align*}
					\mu_t^{(k)} = 	\begin{cases}
						\frac{1}{B} \sum_{i = 1}^B \nabla f(x_0^t; \xi_{t,i}^{(k)})  &  \text{for}~k \in \mathcal{G} \\
						\ast & \text{for}~k \notin \mathcal{G}    	
					\end{cases}	 \qquad	\textbf{$\rightarrow$ Push to SN}\qquad \qquad\qquad \qquad \qquad\qquad \qquad
				\end{align*} \vspace{-0.13 in} }
			$\mu_t^{\text{med}} \leftarrow \mu_t^{(k)}$ where $k \in [K]$ is any WN such that $|\{ k' \in [K]: \| \mu_t^{(k')} - \mu_t^{(k)} \| \leq \mathfrak{T}_\mu \}| > K/2;$
			
			$\mathcal{G}_t = \{k \in [K] : \| \mu_t^{(k)} - \mu_t^{\text{med}} \| \leq 2 \mathfrak{T}_\mu \};$
			
			\If{$|\mathcal{G}_t| < (1 - \alpha)K$}{$\mu_t^{\text{med}} \leftarrow \mu_t^{(k)}$ where $k \in [K]$ is any WN s.t. $|\{ k' \in [K]: \| \mu_t^{(k')} - \mu_t^{(k)} \| \leq 2 \mathcal{V} \}| > K/2;$
				
				$\mathcal{G}_t = \{k \in [K] : \| \mu_t^{(k)} - \mu_t^{\text{med}} \| \leq 4 \mathcal{V} \}$;}
			$\mu_t = \frac{1}{|\mathcal{G}_t|} \sum_{k \in \mathcal{G}_t} \mu_t^{(k)}$;
			
			\For{$n = 1,2,\ldots,N_t$, $N_t \sim \text{Geom}(\frac{B}{B + 1})$}
			{$v_{n-1}^t = \nabla f(x_{n-1}^t ; \xi_{n-1}^t)  - \nabla f(x_0^t ; \xi_{n-1}^t) + \mu_t$; 
				
				$x_n^t = x_{n-1}^t - \eta_t v_{n-1}^t$;
			}
			
			$\tilde{x}_t \leftarrow x_{N_t}^t$;
		}
		\KwOut{$\tilde{x}_a$ chosen uniformly randomly from $(\tilde{x}_t )_{t =1}^T$.}
		\caption{Byzantine SVRG with Distributed Batch Gradient Computations}
		\label{alg1}
	\end{algorithm}
	Note that the above rule is motivated by the fact that for good nodes $k \in \mathcal{G}$, we expect with high probability the batch gradients to be concentrated around the true gradient, $\nabla f(\cdot)$. However, with non-zero probability, the above set can be empty or can have $|\mathcal{G}_t|  < (1 - \alpha) K$. As we know that $|\mathcal{G}| \geq (1 - \alpha)K$, we would want to avoid such scenarios. For this purpose, if we have $|\mathcal{G}_t|  < (1 - \alpha) K$ using the first rule (see Algorithm \ref{alg1}), we update the definition of the median using $2 \mathcal{V}$ to define the new median and construct: 
	$\mathcal{G}_t = \{k \in [K] : \| \mu_t^{(k)} - \mu_t^{\text{med}} \| \leq 4 \mathcal{V} \}$. This condition ensures that we always have $|\mathcal{G}_t| \geq (1 - \alpha) K$.  Note that the Byzantine filtering rule is similar to the one used in \citet{Alistarh_NIPS_2018}. However, in contrast to \citet{Alistarh_NIPS_2018}, we do not maintain a running sum of any statistic to filter Byzantines. Instead, we control the impact of Byzantine nodes by selecting a batch size appropriately. Importantly, \citet{Alistarh_NIPS_2018} considered only convex objectives.	Next, we provide the guarantees for the algorithm.
	\section{Convergence Guarantees}
	\label{Sec: Conv}
	Let us denote by $\mathbb{E}G_{\text{comp, SN}}(\epsilon)$ the expected number of total gradient computations required at the SN (same number of computations are required at individual WNs) to reach an $\epsilon$-stationary point. By Algorithm \ref{alg1} we have: 
	$$\mathbb{E}G_{\text{comp, SN}}(\epsilon) = \sum_{t = 1}^T (B + \mathbb{E}[N_t]) = 2 T B~~\text{with}~~N_t \sim \text{Geom}\left(\frac{B}{B + 1}\right).$$
	where $T$ is the number of iterations required to reach an $\epsilon$-stationary point (see Output of Algorithm \ref{alg1}). Next, we state the convergence result.
	\begin{theorem}
		\label{Thm:Convergence_BSVRG}
		If Assumptions \ref{Ass: LipCont} and \ref{Ass: BoundedGradVar} are satisfied, for step size $\eta_t = \eta = \frac{1}{3 L B^{{2}/{3}}}$ and $B \geq 16$ such that we have $\delta \in (0,1)$ such that: $\mathrm{e}^{\frac{\delta B}{2(1 - 2 \delta)}} \leq \frac{2K}{\delta} \leq \mathrm{e}^{\frac{B}{2}}~~\text{and}~~ \delta \leq \frac{1}{25 K B}$
		then we have:
		\begin{align*}
			\mathbb{E} \| \nabla f(\tilde{x}_a) \|^2  & \leq   \underbrace{ \frac{ 12 L \mathbb{E}[ f(\tilde{x}_{0}) -   f(\tilde{x}^\ast) ]}{T B^{{1}/{3}}} }_{T = O \left(  \frac{1}{\epsilon B^{{1}/{3}}}  \right)}     +   \underbrace{\frac{32 \mathcal{V}^2}{  (1 - \alpha)^{2} K B }}_{B = O \left(  \frac{1}{\epsilon K}  \right)}  +    \underbrace{\frac{2176 \alpha^2 \mathcal{V}^2 C}{ (1 - \alpha)^{2} B}}_{B = {O} \left(  \frac{\alpha^2}{\epsilon }  \right)}.
		\end{align*}
	\end{theorem}
	Using the above, $\mathbb{E}G_{\text{comp, SN}}(\epsilon)$ can be computed as:
	\begin{corollary}
		Under the assumptions stated in Theorem \ref{Thm:Convergence_BSVRG}:\\
		(i): We have: $\mathbb{E} G_{\text{comp, SN}}(\epsilon)  \leq \tilde{O}\left(\frac{1}{\epsilon^{{5}/{3}}  K^{{2}/{3}}} +  \frac{\alpha^{{4}/{3}}}{ \epsilon^{{5}/{3}} }\right)$ where $\tilde{O}(\cdot)$ hides the $\log$ factors.\\
		(ii) Moreover, when $\alpha = 0$ we have: $\mathbb{E} G_{\text{comp, SN}}(\epsilon)  \leq  O \left(\frac{1}{\epsilon^{{5}/{3}}  K^{{2}/{3}}} \right)$.
	\end{corollary}
	Note that for the case when $\alpha = 0$ and $k \leq 1/\epsilon$, our algorithm improves upon the best known convergence rates of $O\big(\frac{1}{\epsilon^2 K} \big)$ for distributed stochastic non-convex optimization given in \citet{Hao_ICML_2019,Jiang_NIPS_2018}. Moreover, if Algorithm \ref{alg1} is run only at the SN, we achieve $\mathbb{E} G_{\text{comp, SN}}(\epsilon)  \leq  O \left(\frac{1}{\epsilon^{{5}/{3}}} \right)$ which is the same as is computed in \citet{Lei_Jordan_SCSG}. 
	\section{Conclusion}
	In this work, we proposed the first non-convex SVRG like algorithm which considers distributed optimization in the presence of Byzantine nodes. We proposed a novel aggregation rule which is independent of the problem dimension, $d$. Importantly, we captured the effect of Byzantine nodes in the network and showed that the proposed algorithm outperforms the best known convergence rates known in the literature in the presence of Byzantine nodes.
	\bibliography{refs}

\begin{thebibliography}{17}
\providecommand{\natexlab}[1]{#1}
\providecommand{\url}[1]{\texttt{#1}}
\expandafter\ifx\csname urlstyle\endcsname\relax
  \providecommand{\doi}[1]{doi: #1}\else
  \providecommand{\doi}{doi: \begingroup \urlstyle{rm}\Url}\fi

\bibitem[Alistarh et~al.(2018)Alistarh, Allen-Zhu, and Li]{Alistarh_NIPS_2018}
Dan Alistarh, Zeyuan Allen-Zhu, and Jerry Li.
\newblock Byzantine stochastic gradient descent.
\newblock In \emph{Advances in Neural Information Processing Systems 31}, pages
  4613--4623. Curran Associates, Inc., 2018.

\bibitem[Blanchard et~al.(2017)Blanchard, El~Mhamdi, Guerraoui, and
  Stainer]{Blanchard_NIPS_2017}
Peva Blanchard, El~Mahdi El~Mhamdi, Rachid Guerraoui, and Julien Stainer.
\newblock Machine learning with adversaries: Byzantine tolerant gradient
  descent.
\newblock In \emph{Advances in Neural Information Processing Systems 30}, pages
  119--129. Curran Associates, Inc., 2017.

\bibitem[Dekel et~al.(2012)Dekel, Gilad-Bachrach, Shamir, and
  Xiao]{Dekel_JMLR_2012}
Ofer Dekel, Ran Gilad-Bachrach, Ohad Shamir, and Lin Xiao.
\newblock Optimal distributed online prediction using mini-batches.
\newblock \emph{Journal of Machine Learning Research}, 13\penalty0
  (Jan):\penalty0 165--202, 2012.

\bibitem[Ho et~al.(2013)Ho, Cipar, Cui, Lee, Kim, Gibbons, Gibson, Ganger, and
  Xing]{Ho_NIPS_2013}
Qirong Ho, James Cipar, Henggang Cui, Seunghak Lee, Jin~Kyu Kim, Phillip~B
  Gibbons, Garth~A Gibson, Greg Ganger, and Eric~P Xing.
\newblock More effective distributed ml via a stale synchronous parallel
  parameter server.
\newblock In \emph{Advances in neural information processing systems}, pages
  1223--1231, 2013.

\bibitem[Jiang and Agrawal(2018)]{Jiang_NIPS_2018}
Peng Jiang and Gagan Agrawal.
\newblock A linear speedup analysis of distributed deep learning with sparse
  and quantized communication.
\newblock In \emph{Advances in Neural Information Processing Systems}, pages
  2525--2536, 2018.

\bibitem[Lei and Jordan(2017)]{Lei_Jordan_PMLR_2017}
Lihua Lei and Michael Jordan.
\newblock {Less than a Single Pass: Stochastically Controlled Stochastic
  Gradient}.
\newblock In Aarti Singh and Jerry Zhu, editors, \emph{Proceedings of the 20th
  International Conference on Artificial Intelligence and Statistics},
  volume~54 of \emph{Proceedings of Machine Learning Research}, pages 148--156,
  Fort Lauderdale, FL, USA, 20--22 Apr 2017. PMLR.

\bibitem[Lei et~al.(2017)Lei, Ju, Chen, and Jordan]{Lei_Jordan_SCSG}
Lihua Lei, Cheng Ju, Jianbo Chen, and Michael~I Jordan.
\newblock Non-convex finite-sum optimization via scsg methods.
\newblock In \emph{Advances in Neural Information Processing Systems}, pages
  2348--2358, 2017.

\bibitem[Li et~al.(2019)Li, Xu, Chen, Giannakis, and Ling]{Li_AAAI_2019}
Liping Li, Wei Xu, Tianyi Chen, Georgios~B Giannakis, and Qing Ling.
\newblock Rsa: Byzantine-robust stochastic aggregation methods for distributed
  learning from heterogeneous datasets.
\newblock In \emph{Proceedings of the AAAI Conference on Artificial
  Intelligence}, volume~33, pages 1544--1551, 2019.

\bibitem[Recht et~al.(2011)Recht, Re, Wright, and Niu]{Recht_NIPS_2011}
Benjamin Recht, Christopher Re, Stephen Wright, and Feng Niu.
\newblock Hogwild: A lock-free approach to parallelizing stochastic gradient
  descent.
\newblock In \emph{Advances in neural information processing systems}, pages
  693--701, 2011.

\bibitem[Reddi et~al.(2016)Reddi, Hefny, Sra, Poczos, and
  Smola]{Reddy_PMLR_2016}
Sashank~J. Reddi, Ahmed Hefny, Suvrit Sra, Barnabas Poczos, and Alex Smola.
\newblock Stochastic variance reduction for nonconvex optimization.
\newblock In \emph{Proceedings of The 33rd International Conference on Machine
  Learning}, volume~48 of \emph{Proceedings of Machine Learning Research},
  pages 314--323, New York, New York, USA, 20--22 Jun 2016. PMLR.

\bibitem[Su and Xu(2018)]{Su_2018_securing}
Lili Su and Jiaming Xu.
\newblock Securing distributed machine learning in high dimensions.
\newblock \emph{arXiv preprint arXiv:1804.10140}, 2018.

\bibitem[Xie et~al.(2018{\natexlab{a}})Xie, Koyejo, and Gupta]{Xie_2018_phocas}
Cong Xie, Oluwasanmi Koyejo, and Indranil Gupta.
\newblock Phocas: dimensional byzantine-resilient stochastic gradient descent.
\newblock \emph{arXiv preprint arXiv:1805.09682}, 2018{\natexlab{a}}.

\bibitem[Xie et~al.(2018{\natexlab{b}})Xie, Koyejo, and Gupta]{Xie_2018_zeno}
Cong Xie, Oluwasanmi Koyejo, and Indranil Gupta.
\newblock Zeno: Distributed stochastic gradient descent with suspicion-based
  fault-tolerance.
\newblock \emph{arXiv preprint arXiv:1805.10032}, 2018{\natexlab{b}}.

\bibitem[Yin et~al.(2018)Yin, Chen, Kannan, and Bartlett]{Yin_PMLR_2018}
Dong Yin, Yudong Chen, Ramchandran Kannan, and Peter Bartlett.
\newblock {B}yzantine-robust distributed learning: Towards optimal statistical
  rates.
\newblock In \emph{Proceedings of the 35th International Conference on Machine
  Learning}, volume~80 of \emph{Proceedings of Machine Learning Research},
  pages 5650--5659, Stockholmsmässan, Stockholm Sweden, 10--15 Jul 2018. PMLR.

\bibitem[Yin et~al.(2019)Yin, Chen, Kannan, and Bartlett]{Yin_PMLR_2019}
Dong Yin, Yudong Chen, Ramchandran Kannan, and Peter Bartlett.
\newblock Defending against saddle point attack in {B}yzantine-robust
  distributed learning.
\newblock In \emph{Proceedings of the 36th International Conference on Machine
  Learning}, volume~97 of \emph{Proceedings of Machine Learning Research},
  pages 7074--7084, Long Beach, California, USA, 09--15 Jun 2019. PMLR.

\bibitem[Yu et~al.(2019)Yu, Jin, and Yang]{Hao_ICML_2019}
Hao Yu, Rong Jin, and Sen Yang.
\newblock On the linear speedup analysis of communication efficient momentum
  sgd for distributed non-convex optimization.
\newblock In \emph{International Conference on Machine Learning}, pages
  7184--7193, 2019.

\bibitem[Zinkevich et~al.(2010)Zinkevich, Weimer, Li, and
  Smola]{Zinkevich_NIPS_2010}
Martin Zinkevich, Markus Weimer, Lihong Li, and Alex~J Smola.
\newblock Parallelized stochastic gradient descent.
\newblock In \emph{Advances in neural information processing systems}, pages
  2595--2603, 2010.

\end{thebibliography}
	\newpage
	\clearpage
	\appendix
	\section{}
	For the purpose of the proof, we consider a more general model where the nodes can choose different batch sizes, $B_t$, for $t = 1,2, \ldots, T$, across epochs. For varying batch sizes in Algorithm \ref{alg1} the Byzantine filtering constant $\mathfrak{T}_\mu$ will be evaluated inside the epochs and will be a function of $B_t$.  Rest of the algorithm will stay the same. Our proof follows the structure similar to the one in \citet{Lei_Jordan_SCSG} with a few major differences. The problem considered in  \citet{Lei_Jordan_SCSG} was a finite sum problem with all the gradient computations designated at the central node. The key novelty in our proof lies in proving the boundedness of the norm square of the error term $e_t$ (see Lemma \ref{lem: e_t}).    \vspace{0.1 in}\\
	{\bf Proof of Theorem \ref{Thm:Convergence_BSVRG}:} Using the smoothness of function $f$ we have:
	\begin{align}
		\mathbb{E}_{\xi_n^t}  f(x_{n + 1}^t)  & \leq f(x_n^t) - \eta_t \langle  \mathbb{E}_{\xi_n^t} v_n^t ,  \nabla f(x_n^t)   \rangle + \frac{L \eta_t^2}{2} \mathbb{E}_{\xi_n^t} \| v_n^t \|^2 \nonumber\\
		& \overset{(a)}{\leq}   f(x_n^t) - \eta_t (1 - L\eta_t) \| \nabla f(x_n^t) \|^2  \nonumber\\
		&  \qquad \qquad \qquad -  \eta_t      \langle  e_t  ,  \nabla f(x_n^t)   \rangle + \frac{L^3 \eta_t^2}{2} \| x_n^t - x_0^t \|^2 + L \eta_t^2 \| e_t \|^2
		\label{eq: For_Finite_T1}
	\end{align}
	where $(a)$ follows from Lemma \ref{lem: Norm_v}. Denoting by $\mathbb{E}_t$ the expectation w.r.t. all $\xi_1^t, \xi_2^t, \ldots$ given $N_t$. Since  $\xi_1^t, \xi_2^t, \ldots$  are independent of $N_t$, $\mathbb{E}_t$ is equivalent to expectation w.r.t. $\xi_1^t, \xi_2^t, \ldots$. We have from above
	\begin{align*}
		\mathbb{E}_{t}  f(x_{n + 1}^t)  & \leq   \mathbb{E}_{t}  f(x_n^t) - \eta_t (1 - L\eta_t)  \mathbb{E}_{t} \| \nabla f(x_n^t) \|^2  \\
		& \qquad \qquad -  \eta_t   \mathbb{E}_{t}    \langle  e_t  ,  \nabla f(x_n^t)   \rangle + \frac{L^3 \eta_t^2}{2} \mathbb{E}_{t} \| x_n^t - x_0^t \|^2 + L \eta_t^2 \| e_t \|^2
	\end{align*}
	Now taking $n = N_t$ and denoting by $\mathbb{E}_{N_t}$ expectation w.r.t. $N_t$ we get: 
	\begin{align*}
		\mathbb{E}_{N_t} 	\mathbb{E}_{t}  f(x_{N_t + 1}^t)  & \leq  	\mathbb{E}_{N_t}   \mathbb{E}_{t}  f(x_{N_t}^t) - \eta_t (1 - L\eta_t)  \mathbb{E}_{N_t} \mathbb{E}_{t} \| \nabla f(x_{N_t}^t) \|^2  \\
		& \qquad \qquad -  \eta_t   	\mathbb{E}_{N_t}   \mathbb{E}_{t}    \langle  e_t  ,  \nabla f(x_{N_t}^t)   \rangle + \frac{L^3 \eta_t^2}{2}  	\mathbb{E}_{N_t}  \mathbb{E}_{t} \| x_{N_t}^t - x_0^t \|^2 + L \eta_t^2 \| e_t \|^2
	\end{align*}
	Using Fubini's theorem and rearranging the terms we have
	\begin{align*}
		\eta_t (1 - L\eta_t) \mathbb{E}_{N_t} \mathbb{E}_{t} \| \nabla f(x_{N_t}^t) \|^2   & \leq  	  \mathbb{E}_{N_t}  \mathbb{E}_{t}  f(x_{N_t}^t)  - 	 \mathbb{E}_{N_t} \mathbb{E}_{t}  f(x_{N_t + 1}^t) \\
		& \quad  -  \eta_t   	\mathbb{E}_{N_t}   \mathbb{E}_{t}    \langle  e_t  ,  \nabla f(x_{N_t}^t)   \rangle + \frac{L^3 \eta_t^2}{2}  	\mathbb{E}_{N_t}  \mathbb{E}_{t} \| x_{N_t}^t - x_0^t \|^2 + L \eta_t^2 \| e_t \|^2 \\
		& \overset{(a)}{=} \frac{1}{B_t} [ f(x_0^t) -  \mathbb{E}_t \mathbb{E}_{N_t} f(x_{N_t}^t) ] \\
		& \quad  -  \eta_t   	\mathbb{E}_{N_t}   \mathbb{E}_{t}    \langle  e_t  ,  \nabla f(x_{N_t}^t)   \rangle + \frac{L^3 \eta_t^2}{2}  	\mathbb{E}_{N_t}  \mathbb{E}_{t} \| x_{N_t}^t - x_0^t \|^2 + L \eta_t^2 \| e_t \|^2 
	\end{align*}
	where $(a)$ follows from Lemma \ref{lem: GeomRV}, Lemma \ref{lem: Finite} and Fubini's theorem. Now taking expectation over all the randomness and using the fact that $x_{N_t}^t = \tilde{x}_t$ and $\tilde{x}_{t - 1} = x_0^t$  we get
	\begin{align*}
		\eta_t (1 - L\eta_t) \mathbb{E} \| \nabla f(\tilde{x}_t) \|^2   &  \leq  \frac{1}{B_t} \mathbb{E}[ f(\tilde{x}_{t- 1}) -   f(\tilde{x}_t) ]   -  \eta_t   	\mathbb{E}    \langle  e_t  ,  \nabla f(\tilde{x}_{t})   \rangle  \\
		& \qquad \qquad\qquad
		+ \frac{L^3 \eta_t^2}{2}  	\mathbb{E} \| \tilde{x}_{t} - \tilde{x}_{t- 1} \|^2 + L \eta_t^2 \mathbb{E} \| e_t \|^2   \\
		& \overset{(a)}{=}    \frac{1}{B_t} \mathbb{E}[ f(\tilde{x}_{t- 1}) -   f(\tilde{x}_t) ]  + \frac{1}{B_t} \mathbb{E} \langle  e_t, \tilde{x}_{t} - \tilde{x}_{t-1}  \rangle  \\
		& \quad \qquad\qquad +  \frac{L^3 \eta_t^2}{2}  	\mathbb{E} \| \tilde{x}_{t} - \tilde{x}_{t- 1} \|^2 +    \eta_t (1 + L \eta_t) \mathbb{E} \| e_t \|^2 \\
		&  \overset{(b)}{\leq}    \frac{1}{B_t} \mathbb{E}[ f(\tilde{x}_{t- 1}) -   f(\tilde{x}_t) ]  + \left( \frac{1}{2 \eta_t B_t} \left( - \frac{1}{B_t} + \eta_t^2 L^2 \right) \right) \mathbb{E} \| \tilde{x}_t - \tilde{x}_{t-1} \|^2 \\
		&  \qquad  \qquad     - \frac{1}{B_t} \mathbb{E} \langle   \nabla f(\tilde{x}_t) , \tilde{x}_t - \tilde{x}_{t-1}  \rangle + \frac{\eta_t}{B_t} \mathbb{E} \| \nabla f(\tilde{x}_{t}) \|^2 + \frac{\eta_t}{B_t} \mathbb{E} \| e_t \|^2 \\
		&  \qquad  \qquad \qquad +  \frac{L^3 \eta_t^2}{2}  	\mathbb{E} \| \tilde{x}_{t} - \tilde{x}_{t- 1} \|^2 +    \eta_t (1 + L \eta_t) \mathbb{E} \| e_t \|^2 
	\end{align*}
	where $(a)$ follows from Lemma \ref{lem: InnerProd}, $(b)$ follows from Lemma \ref{lem: Norm_Sq}. Rearranging the terms to get:
	\begin{align*}
		& \eta_t \left(   1 - L \eta_t -  \frac{1}{B_t} \right)   \mathbb{E} \| \nabla f(\tilde{x}_t) \|^2   +  \left(  \frac{1 - \eta_t^2 L^2 B_t - \eta_t^3 L^3  B_t^2}{2 \eta_t B_t^2}   \right)  \mathbb{E} \| \tilde{x}_t - \tilde{x}_{t-1} \|^2  \\
		& \qquad \qquad \leq \frac{1}{B_t}   \mathbb{E}[ f(\tilde{x}_{t- 1}) -   f(\tilde{x}_t) ]  + \frac{1}{B_t}  \mathbb{E} \langle   \nabla f(\tilde{x}_t) , \tilde{x}_{t-1} - \tilde{x}_{t}  \rangle  + \eta_t \left( 1 + L \eta_t + \frac{1}{B_t}  \right)  \mathbb{E}\| e_t \|^2.
	\end{align*}
	Now using the Young's inequality $\mathbb{E} \langle  a , b \rangle \leq \mathbb{E} \left[  \frac{\beta}{2} \|a \|^2 + \frac{1}{2 \beta} \| b\|^2  \right]$ on $ \mathbb{E} \langle   \nabla f(\tilde{x}_t) , \tilde{x}_{t-1} - \tilde{x}_{t}  \rangle$ with $$\beta =  \frac{1 - \eta_t^2 L^2 B_t  -  \eta_t^3 L^3 B_t^2 }{\eta_t B_t},~ ~~~a = \tilde{x}_{t-1} - \tilde{x}_t~~~~ \text{and}~~~~b = \nabla f(\tilde{x}_t).$$ 
	we get
	\begin{align}
		\label{eq: Simplify_GradIneq}
		& \eta_t \left(   1 - L \eta_t -  \frac{1}{B_t}   -  \frac{1}{2 (1 - \eta_t^2 L^2 B_t  -  \eta_t^3 L^3 B_t^2) }  \right)   \mathbb{E} \| \nabla f(\tilde{x}_t) \|^2    \nonumber \\
		& \qquad \qquad    \leq \frac{1}{B_t}   \mathbb{E}[ f(\tilde{x}_{t- 1}) -   f(\tilde{x}_t) ]  + \eta_t \left( 1 + L \eta_t + \frac{1}{B_t}  \right)  \mathbb{E}\| e_t \|^2 \nonumber  \\
		& \eta_t \left(   1 - L \eta_t -  \frac{1}{B_t}   -  \frac{1}{2 (1 - \eta_t^2 L^2 B_t  -  \eta_t^3 L^3 B_t^2) }  \right)   \mathbb{E} \| \nabla f(\tilde{x}_t) \|^2    \nonumber  \\
		&   \qquad \qquad  \overset{(a)}{\leq} \frac{1}{B_t}   \mathbb{E}[ f(\tilde{x}_{t- 1}) -   f(\tilde{x}_t) ]  + \eta_t \left( 1 + L \eta_t + \frac{1}{B_t}  \right) \left( \frac{4 \mathcal{V}^2}{(1 - \alpha)^2 K B_t} + \frac{272 \alpha^2  \mathcal{V}^2 C}{(1 - \alpha)^2 B_t} \right).
	\end{align}
	where $(a)$ follows from Lemma \ref{lem: e_t}. 
	
	Choosing $\eta_t$ such that we have $1 - \eta_t^2 L^2 B_t  -  \eta_t^3 L^3 B_t^2 > 0$, this implies that we have 
	$$\frac{1}{2 (1 - \eta_t^2 L^2 B_t  -  \eta_t^3 L^3 B_t^2) }  > \frac{1}{2}.$$
	Further we choose $\eta_t$ such that we have: 
	\begin{align*}
		1 - L \eta_t -  \frac{1}{B_t}   -  \frac{1}{2 (1 - \eta_t^2 L^2 B_t  -  \eta_t^3 L^3 B_t^2) }  \geq \frac{1}{4} \\
		L \eta_t +  \frac{1}{B_t}   +  \frac{1}{2 (1 - \eta_t^2 L^2 B_t  -  \eta_t^3 L^3 B_t^2) }
		\leq \frac{3}{4}.
	\end{align*}
	Choosing $\eta_t$ such that we can ensure:
	$$\textit{(i)}:~ \frac{1}{2} < \frac{1}{2 (1 - \eta_t^2 L^2 B_t  -  \eta_t^3 L^3 B_t^2) }  \leq \frac{5}{8},~~\textit{(ii)}:~ L \eta_t \leq \frac{1}{16}~~~~\text{and}~~\textit{(iii)}:~ \frac{1}{B_t} \leq \frac{1}{16}. $$
	Condition \textit{(i)} above implies:
	$$ \eta_t^2 L^2 B_t  +  \eta_t^3 L^3 B_t^2  \leq \frac{1}{5}.$$
	Further ensuring $\eta_t$ such that 
	$$ \eta_t^2 L^2 B_t \leq \frac{1}{10}~~~\text{and}~~~  \eta_t^3 L^3 B_t^2 \leq \frac{1}{10}.$$
	This implies that 
	$$\eta_t \leq \frac{1}{10^{{1}/{2}} L B_t^{{1}/{2}}}~~\text{and}~~ \eta_t \leq \frac{1}{10^{{1}/{3}} L B_t^{{2}/{3}}}.$$ 
	Furthermore from conditions \textit{(ii)} and \textit{(iii)} above we get: 
	$$ \eta_t \leq \frac{1}{16L} ~~~\text{and}~~~B_t \geq 16.$$
	The above discussion implies that we must have $B_t \geq 16$ and we can choose $\eta_t \leq \frac{1}{3 L B_t^{{2}/{3}}}$ as we have 
	$$\frac{1}{3 L B_t^{{2}/{3}}}  \leq  \min \left\{\frac{1}{16L} , \frac{1}{10^{{1}/{3}} L B_t^{{2}/{3}}},  \frac{1}{10^{{1}/{2}} L B_t^{{1}/{2}}} \right\}~\text{for}~B_t \geq 16.$$
	This choice of $\eta_t$ ensures that the term: 
	\begin{align}
		\label{eq: LHS_Term}
		1 - L \eta_t -  \frac{1}{B_t}   -  \frac{1}{2 (1 - \eta_t^2 L^2 B_t  -  \eta_t^3 L^3 B_t^2) }  \geq \frac{1}{4}. 
	\end{align}
	Now replacing $\eta_t$ and $B_t$ in the term:
	\begin{align}
		\label{eq: RHS_Term}
		1 +  \eta_t L + \frac{1}{B_t}  \leq 1 + \frac{1}{3 B_t^{{2}/{3}}} + \frac{1}{B_t} \leq 2.
	\end{align}
	Now replacing \eqref{eq: LHS_Term} and \eqref{eq: RHS_Term} in \eqref{eq: Simplify_GradIneq}, and replacing $\eta_t = \frac{1}{3 L B_t^{{2}/{3}}}$ we get
	\begin{align*}
		\frac{\eta_t}{4}     \mathbb{E} \| \nabla f(\tilde{x}_t) \|^2       & \leq \frac{1}{B_t}   \mathbb{E}[ f(\tilde{x}_{t- 1}) -   f(\tilde{x}_t) ]  + 2 \eta_t\left(	\frac{4 \mathcal{V}^2}{(1 - \alpha)^2 K B_t} + \frac{272 \alpha^2 \mathcal{V}^2 C}{(1 - \alpha)^2 B_t} \right) \\
		\mathbb{E} \| \nabla f(\tilde{x}_t) \|^2  &  \leq  \frac{12 L  \mathbb{E}[ f(\tilde{x}_{t- 1}) -   f(\tilde{x}_t) ]}{B_t^{{1}/{3}}}  +   	\frac{32 \mathcal{V}^2}{(1 - \alpha)^2 K B_t } + \frac{2176 \alpha^2 \mathcal{V}^2 C}{(1 - \alpha)^2 B_t }  .
	\end{align*}
	{\bf Constant batch size:} For $B_t = B$ and summing over $t = 1,2, \ldots, T$ and choosing $x_a$ using Algorithm \ref{alg1} we get:
	\begin{align*}
		& \mathbb{E} \| \nabla f(\tilde{x}_a) \|^2  \\
		&  \leq    \frac{ 12 L \mathbb{E}[ f(\tilde{x}_{0}) -   f(\tilde{x}^\ast) ]      +   32 \mathcal{V}^2 (1 - \alpha)^{-2} K^{-1} \sum_{t = 1}^T B_t^{{-2}/{3}}  + 2176 \alpha^2 \mathcal{V}^2 C (1 - \alpha)^{-2} \sum_{t = 1}^T B_t^{{-2}/{3}}    }{\sum_{t = 1}^T B_t^{{1}/{3}}} .
	\end{align*}  
	Replace $B_t = B$ we get:
	\begin{align*}
		\mathbb{E} \| \nabla f(\tilde{x}_a) \|^2  & \leq   \underbrace{ \frac{ 12 L \mathbb{E}[ f(\tilde{x}_{0}) -   f(\tilde{x}^\ast) ]}{T B^{{1}/{3}}} }_{T = O \left(  \frac{1}{\epsilon B^{{1}/{3}}}  \right)}     +   \underbrace{\frac{32 \mathcal{V}^2}{  (1 - \alpha)^{2} K B }}_{B = O \left(  \frac{1}{\epsilon K}  \right)}  +    \underbrace{\frac{2176 \alpha^2 \mathcal{V}^2 C}{ (1 - \alpha)^{2} B}}_{B = {O} \left(  \frac{\alpha^2}{\epsilon }  \right)}.
	\end{align*}
	Now to guarantee that we get an $\epsilon$-accurate solution we need: 
	$T = O \left(  \frac{1}{\epsilon B^{{1}/{3}}}  \right)$ number of iterations for the first term. For the second term, we need $B_K = O \left(  \frac{1}{\epsilon K}  \right)$ batch size and for the third term, we need $B_\alpha = {O} \left(  \frac{\alpha^2}{\epsilon }  \right)$ batch size to account for the Byzantine workers. 
	This implies that the total number of gradient computations, $\mathbb{E}G_{\text{comp, SN}}$ required at the SN (and at the individual WNs) on an average are of the order of: 
	\begin{align*}
		\mathbb{E} G_{\text{comp, SN}}(\epsilon) & \leq T B_K + T B_\alpha \\
		\mathbb{E} G_{\text{comp, SN}}(\epsilon) & \leq O\left(\frac{1}{\epsilon^{{5}/{3}}  K^{{2}/{3}}} +  \frac{\alpha^{{4}/{3}}}{ \epsilon^{{5}/{3}} }\right).
	\end{align*}
	And the expected number of gradient computations across the network denoted by, $\mathbb{E}G_{\text{comp, NW}}$,  are of the order of:
	\begin{align*}
		\mathbb{E} G_{\text{comp, NW}}(\epsilon) & \leq O\left(\frac{K^{1/3}}{\epsilon^{{5}/{3}} } +  \frac{K \alpha^{{4}/{3}}}{ \epsilon^{{5}/{3}} }\right).
	\end{align*}
	
	Moreover, note that if we have $\alpha = 0$ we get the expected computational complexity and the expected number of gradient computations across the network are of the order of:  
	\begin{align*}
		\mathbb{E} G_{\text{comp, SN}}(\epsilon)  \leq O\left(\frac{1}{\epsilon^{{5}/{3}}  K^{{2}/{3}}}\right)\quad \text{and}\quad \mathbb{E} G_{\text{comp, NW}}(\epsilon)  \leq O\left(\frac{K^{1/3}}{\epsilon^{{5}/{3}} }\right)    . 
	\end{align*}\hfill	$\blacksquare$

	\begin{lemma} \label{lem: Norm_v}
		We have
		$$\mathbb{E}_{\xi_n^t} \| v_n^t \|^2   \leq L^2 \| x_n^t - x_0^t \|^2  + 2 \| \nabla f(x_n^t) \|^2 + 2 \| e_t \|^2.$$
	\end{lemma}
	\begin{proof}
		From the definition of $v_n^t$ we have:
		$$ v_n^t = \nabla f(x_n^t; \xi_n^t) - \nabla f(x_0^t; \xi_n^t) + \mu_t .$$
		where $\mu_t = \frac{1}{|\mathcal{G}_t|}  \sum_{k \in \mathcal{G}_t}  \mu_t^{(k)}$. We define 
		$$e_t = \mu_t - \nabla f(x_0^t) =  \mu_t - \nabla f(\tilde{x}_{t-1}).$$
		This implies that we have 
		$$\mathbb{E}_{\xi_n^t} v_n^t = \nabla f(x_n^t) + e_t.$$
		Now taking $\mathbb{E}_{\xi_n^t} \| v_n^t \|^2$ and using $\mathbb{E} \| Z\|^2 = \mathbb{E} \| Z - \mathbb{E}Z \|^2 + \| \mathbb{E}Z \|^2$, we have
		\begin{align*}
			\mathbb{E}_{\xi_n^t} \| v_n^t \|^2 & =  \mathbb{E}_{\xi_n^t}   \| v_n^t -  \mathbb{E}_{\xi_n^t} v_n^t\|^2 + \| \mathbb{E}_{\xi_n^t} v_n^t \|^2 \\
			& = \mathbb{E}_{\xi_n^t}  \|   \nabla f(x_n^t; \xi_n^t) - \nabla f(x_0^t ; \xi_n^t) - (\nabla f(x_n^t) - \nabla f(x_0^t)  )  \|^2 +  \| \nabla f(x_n^t) + e_t \|^2\\
			& \overset{(a)}{\leq} \mathbb{E}_{\xi_n^t}  \|   \nabla f(x_n^t; \xi_n^t) - \nabla f(x_0^t ; \xi_n^t) - (\nabla f(x_n^t) - \nabla f(x_0^t)  )  \|^2 +  2 \| \nabla f(x_n^t) \|^2 + 2 \|e_t \|^2 \\
			& \overset{(b)}{\leq} \mathbb{E}_{\xi_n^t}  \|   \nabla f(x_n^t; \xi_n^t) - \nabla f(x_0^t ; \xi_n^t)  \|^2 +  2 \| \nabla f(x_n^t) \|^2 + 2 \|e_t \|^2 \\
			& \overset{(c)}{\leq} L^2 \| x_n^t - x_0^t \|^2 + 2 \| \nabla f(x_n^t) \|^2 + 2 \|e_t \|^2.
		\end{align*}
		where $(a)$ follows from the definition of $v_n^t$ and Lemma \ref{lem: Norm_Ineq}, $(b)$ follows from variance inequality and $(c)$ follows from the Gradient Lipschitz continuity of $f(\cdot~; \xi_n^t)$.
	\end{proof}

	\begin{lemma} \label{lem: InnerProd} We have
		\begin{align*}
			\eta_t \mathbb{E} \langle e_t , \nabla f(\tilde{x_t})  \rangle =  \frac{1}{B_t} \mathbb{E} \langle  e_t, \tilde{x}_{t - 1} - \tilde{x}_t  \rangle  - \eta_t \mathbb{E} \| e_t \|^2 . 
		\end{align*}
	\end{lemma}
	\begin{proof}
		Consider the term $M_n^t = \langle  e_t , x_n^t - x_0^t  \rangle$. From the definition of $M_n^t$ we have
		\begin{align*}
			M_{n+1}^t - M_n^t &= \langle  e_t , x_{n+1}^t  - x_n^t \rangle = - \eta_t \langle  e_t ,  v_n^t  \rangle.
		\end{align*}
		Taking expectation w.r.t. $\xi_n^t$, we have
		\begin{align*}
			\mathbb{E}_{\xi_n^t} (M_{n+1}^t - M_n^t) & = - \eta_t \langle  e_t ,  	\mathbb{E}_{\xi_n^t}  v_n^t  \rangle \\
			& \overset{(a)}{=}  - \eta_t \langle  e_t ,     \nabla f(x_n^t)  \rangle - \eta_t \| e_t \|^2,
		\end{align*}
		where $(a)$ follows from the definition of $v_n^t$. Denoting by $\mathbb{E}_t$ the expectation w.r.t. all $\xi_1^t, \xi_2^t, \ldots$ given $N_t$. Since  $\xi_1^t, \xi_2^t, \ldots$  are independent of $N_t$, $\mathbb{E}_t$ is equivalent to expectation w.r.t. $\xi_1^t, \xi_2^t, \ldots$. We have  
		\begin{align*}
			\mathbb{E}_{t} (M_{n+1}^t - M_n^t) & =   - \eta_t \langle  e_t ,    \mathbb{E}_t  \nabla f(x_n^t)  \rangle - \eta_t \| e_t \|^2.
		\end{align*}
		Taking $n = N_t$ and expectation w.r.t. $N_t$ as $\mathbb{E}_{N_t}$ we have
		\begin{align*}
			\mathbb{E}_{N_t}  \mathbb{E}_{t} (M_{N_t+1}^t - M_{N_t}^t)   & =   - \eta_t \langle  e_t ,   \mathbb{E}_{N_t} \mathbb{E}_t  \nabla f(x_{N_t}^t)  \rangle - \eta_t \| e_t \|^2.
		\end{align*}
		Using Fubini's theorem, Lemma \ref{lem: GeomRV}, Lemma \ref{lem: Finite} and using the fact $x_{N_t}^t = \tilde{x}_t$ and $\tilde{x}_{t - 1} = x_0^t$,  we have
		\begin{align*}
			\frac{1}{B_t}	\mathbb{E}_{N_t}  \mathbb{E}_{t}   \langle  e_t ,  \tilde{x}_t - \tilde{x}_{t - 1}  \rangle   	& =   - \eta_t \langle  e_t ,   \mathbb{E}_{N_t} \mathbb{E}_t  \nabla f(\tilde{x}_t)  \rangle - \eta_t \| e_t \|^2.
		\end{align*}
		Taking expectation w.r.t. the whole past yields the statement of the lemma. 
	\end{proof}
	
	\begin{lemma}\label{lem: Norm_Sq}
		We have
		\begin{align*}
			2 \eta_t \mathbb{E} \langle  e_t,  \tilde{x}_t - \tilde{x}_{t - 1} \rangle  & \leq \left( -  \frac{1}{B_t}  + \eta_t^2 L^2 \right) \mathbb{E} \| \tilde{x}_t - \tilde{x}_{t - 1} \|^2    \\
			& \qquad - 2 \eta_t \mathbb{E} \langle  \nabla f(\tilde{x}_t), \tilde{x}_t - \tilde{x}_{t - 1} \rangle + 2 \eta_t^2 \mathbb{E} \| \nabla f(\tilde{x}_t) \|^2 + 2 \eta_t^2 \mathbb{E} \| e_t \|^2. 
		\end{align*}
	\end{lemma}
	\begin{proof}
		We have from the update equation $x_{n+1}^t  = x_n^t - \eta_t v_n^t $, we have
		\begin{align}
			\mathbb{E}_{\xi_n^t} \| x_{n + 1}^t - x_0^t \|^2 & =  \mathbb{E}_{\xi_n^t}  \| x_n^t - \eta_t v_n^t  - x_0^t \|^2 \nonumber\\
			& =   \| x_n^t  - x_0^t  \|^2 + \eta_t^2 \mathbb{E}_{\xi_n^t}  \| v_n^t \|^2 - 2\eta_t \langle  \mathbb{E}_{\xi_n^t}   v_n^t, x_n^t - x_0^t  \rangle \nonumber\\
			& \overset{(a)}{\leq}   (1 + \eta_t^2 L^2)  \| x_n^t  - x_0^t  \|^2  -  2\eta_t \langle  \nabla f(x_n^t)   , x_n^t - x_0^t  \rangle \nonumber\\
			& \qquad \qquad \qquad  - 2 \eta_t \langle   e_t , x_n^t - x_0^t   \rangle + 2 \eta_t^2 \| \nabla f(x_n^t) \|^2  + 2 \eta_t^2 \|e_t \|^2 
			\label{eq: For_Finite_L5}
		\end{align}
		where $(a)$ follows from Lemma \ref{lem: Norm_v} and the definition of $v_n^t$. Denoting by $\mathbb{E}_t$ the expectation w.r.t. all $\xi_1^t, \xi_2^t, \ldots$ given $N_t$. Since  $\xi_1^t, \xi_2^t, \ldots$  are independent of $N_t$, $\mathbb{E}_t$ is equivalent to expectation w.r.t. $\xi_1^t, \xi_2^t, \ldots$. We have  
		\begin{align*}
			\mathbb{E}_{t} \| x_{n + 1}^t - x_0^t \|^2  & \leq  (1 + \eta_t^2 L^2) \mathbb{E}_{t}  \| x_n^t  - x_0^t  \|^2  -  2\eta_t \mathbb{E}_{t} \langle  \nabla f(x_n^t)   , x_n^t - x_0^t  \rangle \\
			& \qquad \qquad \qquad  - 2 \eta_t  \mathbb{E}_{t} \langle   e_t , x_n^t - x_0^t   \rangle + 2 \eta_t^2 \mathbb{E}_{t} \| \nabla f(x_n^t) \|^2  + 2 \eta_t^2 \|e_t \|^2 .
		\end{align*}
		Now taking $n = N_t$ and taking expectation $\mathbb{E}_{N_t}$ w.r.t. $N_t$ we have
		\begin{align*}
			2 \eta_t \mathbb{E}_{N_t}   \mathbb{E}_{t} \langle   e_t , \tilde{x}_t - \tilde{x}_{t-1}   \rangle   &  \leq   (1 + \eta_t^2 L^2)  \mathbb{E}_{N_t}  \mathbb{E}_{t}  \| x_{N_t}^t  - x_0^t  \|^2 - \mathbb{E}_{N_t} \mathbb{E}_{t} \| x_{N_t + 1}^t - x_0^t \|^2  \\
			&  \qquad \qquad -  2\eta_t \mathbb{E}_{N_t}  \mathbb{E}_{t} \langle  \nabla f(\tilde{x}_t)   , \tilde{x}_t - \tilde{x}_{t-1}  \rangle  
			+ 2 \eta_t^2  \mathbb{E}_{N_t}  \mathbb{E}_{t} \| \nabla f(\tilde{x}_t) \|^2  + 2 \eta_t^2 \|e_t \|^2 \\
			&  \overset{(a)}{=}  \left(- \frac{1}{B_t} +  \eta_t^2 L^2 \right)  \mathbb{E}_{N_t}  \mathbb{E}_{t}  \| \tilde{x}_t  - \tilde{x}_{t-1}  \|^2  \\
			&\qquad \qquad -  2\eta_t \mathbb{E}_{N_t}  \mathbb{E}_{t} \langle  \nabla f(\tilde{x}_t)   , \tilde{x}_t - \tilde{x}_{t-1}  \rangle             + 2 \eta_t^2  \mathbb{E}_{N_t}  \mathbb{E}_{t} \| \nabla f(\tilde{x}_t) \|^2  + 2 \eta_t^2 \|e_t \|^2.
		\end{align*}
		where $(a)$ follows from Lemma \ref{lem: GeomRV}, Lemma \ref{lem: Finite}  and Fubini's theorem. Finally, rearranging the terms and taking expectation w.r.t. the whole past yields the lemma. 
	\end{proof}
	
	\begin{lemma}\label{lem: e_t} 
		Choosing $\delta$ and $B_t$ in Algorithm \ref{alg1} such that the following are satisfied: 
		\begin{enumerate}[(i)]
			\item	$\mathrm{e}^{\frac{\delta B_t}{2(1 - 2 \delta)}} \leq \frac{2K}{\delta} \leq \mathrm{e}^{\frac{B_t}{2}}.$ 
			\item 	$\delta \leq \frac{1}{25 K B_t}$.
		\end{enumerate}
		then $\mathbb{E} \| e_t \|^2 $ is bounded as: 
		$$\mathbb{E} \| e_t \|^2 \leq \frac{4 \mathcal{V}^2}{(1 - \alpha)^2 K B_t} + \frac{272 \alpha^2  \mathcal{V}^2 C}{(1 - \alpha)^2 B_t}.$$
	\end{lemma}
	\begin{proof}
		From the definition $e_t = \mu_t - \nabla f(\tilde{x}_{t-1})$, where 
		$\mu_t = \frac{1}{|\mathcal{G}_t|} \sum_{k \in \mathcal{G}_t} \mu_t^{(k)}$. Therefore, we have
		\begin{align*}
			\mathbb{E} \| e_t \|^2 & = \mathbb{E} \bigg\| \frac{1}{|\mathcal{G}_t|} \sum_{k \in \mathcal{G}_t} \left( \mu_t^{(k)} - \nabla f(\tilde{x}_{t-1}) \right) \bigg\|^2
		\end{align*}
		Now let us define three types of events and their complements:
		\begin{definition}
			\label{def: Events}
			We define three events and their complements which will be used to bound $\mathbb{E} \| e_t \|^2$:
			\begin{enumerate}
				\item {\bf Event $A$ (Event $A!$):} We denote the event of Lemma \ref{lem: T_mu} as Event $A$, we define it again here for convenience. For all good nodes $k \in \mathcal{G}$ we have: 
				\begin{enumerate}
					\item $\|\mu_t^{(k)} - \nabla f(\tilde{x}_{t-1})\| \leq \mathcal{V} \sqrt{ \frac{C}{B_t}}$.
					\item $\|\mu_t^{(k)} - \mu_t^{\text{med}}\| \leq 4 \mathcal{V} \sqrt{\frac{C}{B_t}}$ and $\|\mu_t^{\text{med}} - \nabla f(\tilde{x}_{t-1}) \| \leq 3 \mathcal{V} \sqrt{ \frac{C}{B_t}}$.
				\end{enumerate}
				Note that we have $\mathbb{P}[\text{Event}~A]  \geq 1 - \delta$. The complement of Event $A$ is denoted as Event $A!$ and we have $\mathbb{P}[\text{Event}~A!]  \leq \delta$.
				\item  {\bf Event $\bar{A}$ (Event $\bar{A}!$):} We define Event $\bar{A}$ as $\mathcal{G} \subset \mathcal{G}_t$. Consequently, Event $\bar{A}!$ is defined as $\mathcal{G} \not\subset \mathcal{G}_t$. 
				\item  {\bf Event $R1$ (Event $R2$):} We define Event $R1$ as the event that Rule 1 is executed. And Event $R2$ is the complement of Rule $R1$ and indicates that Rule 2 is executed. Rule 1 and Rule 2 are defined below for convenience. 
			\end{enumerate}
		\end{definition}
		
		{\bf Rule 1:} $\mathcal{G}_t = \{k \in [K] : \| \mu_t^{(k)} - \mu_t^{\text{med}} \| \leq 2 \mathfrak{T}_\mu \};$ for median evaluated as:  
		
		$\mu_t^{\text{med}} \leftarrow \mu_t^{(k)}$ where $k \in [K]$ is any WN such that $|\{ k' \in [K]: \| \mu_t^{(k')} - \mu_t^{(k)} \| \leq \mathfrak{T}_\mu \}| > K/2;$ \vspace{0.1 in} 
		
		In case $|\mathcal{G}_t| < (1 - \alpha)K$ we use: \vspace{0.1 in} 
		
		{\bf Rule 2:} $\mathcal{G}_t = \{k \in [K] : \| \mu_t^{(k)} - \mu_t^{\text{med}} \| \leq 4 \mathcal{V} \}$  for median evaluated as:
		
		$\mu_t^{\text{med}} \leftarrow \mu_t^{(k)}$ where $k \in [K]$ is any WN s.t. $|\{ k' \in [K]: \| \mu_t^{(k')} - \mu_t^{(k)} \| \leq 2 \mathcal{V} \}| > K/2;$ \vspace{0.1 in}\\
		{\bf Relationship between events:}
		\begin{itemize}
			\item Note from Lemma \ref{lem: T_mu} and from Definition \ref{def: Events} of Event $A$ we have: Event $A$ $\subset$ Event $\bar{A}$. This and Lemma \ref{lem: T_mu} imply that we have:
			\begin{align}\label{eq: Abar}
				\mathbb{P}[\text{Event} ~\bar{A} ] \geq \mathbb{P}[\text{Event} ~A ] \geq 1 - \delta. 
			\end{align}
			\item Consider Event $\bar{A}!$ as given in Definition \ref{def: Events}. From above we have Event $A!$ $\supset$ Event $\bar{A}!$. This along with Lemma \ref{lem: T_mu} imply that we have:
			\begin{align}\label{eq: Abar!}
				\mathbb{P}[\text{Event} ~\bar{A}! ] \leq \mathbb{P}[\text{Event} ~A! ] \leq  \delta. 
			\end{align}
			\item From the definition of Event $R1$ and Event $A$ in Definition \ref{def: Events} we have: Event $A$ $\subset$ Event $R1$. This further implies from Lemma \ref{lem: T_mu} that we have:
			\begin{align}\label{eq: R1}
				\mathbb{P}[\text{Event}~R1] \geq \mathbb{P}[\text{Event}~A] \geq 1 - \delta
			\end{align}
			\item Event $R2$ is complement of Event $R1$. This and the above implies Event $R2$ $\subset$ Event $A!$. Thus we have from Lemma \ref{lem: T_mu}:
			\begin{align}\label{eq: R2}
				\mathbb{P}[\text{Event}~R2] \leq \mathbb{P}[\text{Event}~A!] \leq \delta
			\end{align}
		\end{itemize}
		Now we can write $	\mathbb{E} \| e_t \|^2$ as:
		\begin{align}
			\mathbb{E} \| e_t \|^2 &\overset{(a)}{=} \mathbb{P}[\text{Event}~\bar{A}]~ \mathbb{E}\left[\| e_t \|^2 | \text{Event}~\bar{A}\right] +  \mathbb{P}[\text{Event}~\bar{A}!]~ \mathbb{E}\left[\| e_t \|^2 | \text{Event}~\bar{A}!\right] \nonumber \\
			& \overset{(b)}{\leq}   \mathbb{E}\left[\| e_t \|^2 | \text{Event}~\bar{A}\right] +  \delta~ \mathbb{E}\left[\| e_t \|^2 | \text{Event}~\bar{A}! \right]
			\label{eq: totalexpectation_et}
		\end{align}
		where $(a)$ follows from the law of total expectation and $(b)$ follows from \eqref{eq: Abar} and \eqref{eq: Abar!} above. Now let us first consider the first term $\mathbb{E}\left[\| e_t \|^2 | \text{Event}~\bar{A}\right]$ under Event $\bar{A}$:
		\begin{align}
			& \mathbb{E}\left[\| e_t \|^2 | \text{Event}~\bar{A}\right] = \mathbb{E}  \left[ \bigg\| \frac{1}{|\mathcal{G}_t|} \sum_{k \in \mathcal{G}_t} \left( \mu_t^{(k)} - \nabla f(\tilde{x}_{t-1}) \right) \bigg\|^2  \bigg| \text{Event}~\bar{A} \right] \nonumber\\
			&      \overset{(a)}{\leq} \frac{1}{(1 - \alpha)^2 K^2}  \mathbb{E} \left[ \bigg\|   \sum_{k \in \mathcal{G}} \left( \mu_t^{(k)} - \nabla f(\tilde{x}_{t-1}) \right) +  \sum_{k \in \mathcal{G}_t \backslash \mathcal{G}} \left( \mu_t^{(k)} - \nabla f(\tilde{x}_{t-1}) \right) \bigg\|^2  \bigg|  \right] ~\text{under Event}~\bar{A} \nonumber\\
			&     \overset{(b)}{\leq} \frac{2}{(1 - \alpha)^2 K^2}  \Bigg( \mathbb{E} \left[ \bigg\|   \sum_{k \in \mathcal{G}} \left( \mu_t^{(k)} - \nabla f(\tilde{x}_{t-1}) \right) \bigg\|^2 \right]  \nonumber\\
			& \qquad \qquad \qquad \qquad \qquad \qquad \qquad + \mathbb{E} \Bigg[ \underbrace{ \bigg\| \sum_{k \in \mathcal{G}_t \backslash \mathcal{G}} \left( \mu_t^{(k)} - \nabla f(\tilde{x}_{t-1}) \right) \bigg\|^2}_{I_{\mathcal{G}_t \backslash \mathcal{G}}} \Bigg] \Bigg) ~\text{under Event}~\bar{A}
			\label{eq: norm_e_t_Abar}
		\end{align}
		where $(a)$ follows from the fact that $|\mathcal{G}_t| \geq (1 - \alpha) K$ and Event $\bar{A}$ implies $\mathcal{G} \subset \mathcal{G}$. $(b)$ follows from Lemma \ref{lem: Norm_Ineq}. Now considering the two terms separately under Event $\bar{A}$, first consider the terms for $k \in \mathcal{G}$:
		\begin{align}
			&	\mathbb{E}   \bigg\|   \sum_{k \in \mathcal{G}} \left( \mu_t^{(k)} - \nabla f(\tilde{x}_{t-1}) \right) \bigg\|^2  =  \mathbb{E}   \bigg\|   \sum_{k \in \mathcal{G}} \left(    \frac{1}{B_t} \sum_{i = 1}^{B_t} \nabla f(\tilde{x}_{t-1} ;  \xi_{t,i}^{(k)})   - \nabla f(\tilde{x}_{t-1}) \right) \bigg\|^2 \nonumber\\
			& \qquad  \overset{(a)}{=} \frac{1}{B_t^2} \mathbb{E} \left[  \sum_{k \in \mathcal{G}}  \sum_{i = 1}^{B_t} \big\|       \nabla f(\tilde{x}_{t-1} ;  \xi_{t,i}^{(k)})   - \nabla f(\tilde{x}_{t-1})   \big\|^2 \right] \nonumber\\
			&  \qquad \quad  + \frac{1}{B_t^2} \underbrace{  \mathbb{E} \left[   \sum_{ (k,i) \neq (k',i'), k,k' \in \mathcal{G}}    \langle   \nabla f(\tilde{x}_{t-1} ;  \xi_{t,i}^{(k)})   - \nabla f(\tilde{x}_{t-1}) ,  \nabla f(\tilde{x}_{t-1} ;  \xi_{t,i'}^{(k')})   - \nabla f(\tilde{x}_{t-1})   \rangle         \right]}_{ = 0}  \nonumber\\ 
			& \qquad  = \frac{1}{B_t^2} \mathbb{E} \left[  \sum_{k \in \mathcal{G}}  \sum_{i = 1}^{B_t} \big\|       \nabla f(\tilde{x}_{t-1} ;  \xi_{t,i}^{(k)})   - \nabla f(\tilde{x}_{t-1})   \big\|^2 \right] \nonumber\\
			&\qquad  \overset{(b)}{\leq} \frac{K \mathcal{V}^2}{B_t} 
			\label{eq: norm_e_t_Abar_Good}
		\end{align}
		where $(a)$ follows from the fact that $f(\tilde{x}_{t-1} ;  \xi_{t,i}^{(k)}) $ are chosen uniformly independently across $i$ and $k$. $(b)$ follows from Assumption \ref{Ass: BoundedGradVar} and the fact that $|\mathcal{G}| \leq K$. Now consider the second term in \eqref{eq: norm_e_t_Abar} defined as:
		$$\mathbb{E}[ I_{\mathcal{G}_t \backslash \mathcal{G}} \big| \text{Event}~\bar{A} ]= \mathbb{E} \left[ \bigg\| \sum_{k \in \mathcal{G}_t \backslash \mathcal{G}} \left( \mu_t^{(k)} - \nabla f(\tilde{x}_{t-1}) \right) \bigg\|^2 \right]~\text{under Event}~\bar{A}$$
		Note that all the nodes $k \in \mathcal{G}_t \backslash \mathcal{G}$ under Event $\bar{A}$ can come either from Rule 1 (Event $R1$ is true)  or from Rule 2 (Event $R2$ is true). Again using the law of total expectation we can write $\mathbb{E}[ I_{\mathcal{G}_t \backslash \mathcal{G}} \big| \text{Event}~\bar{A} ]$ as:
		\begin{align}
			\mathbb{E}[ I_{\mathcal{G}_t \backslash \mathcal{G}} \big| \text{Event}~\bar{A} ] & \overset{(a)}{=} \mathbb{P}[\text{Event}~R1 \big| \text{Event}~\bar{A}]  ~ \mathbb{E}[ I_{\mathcal{G}_t \backslash \mathcal{G}} \big| \text{Event}~\bar{A} , \text{Event}~R1] \nonumber\\
			& \qquad \qquad \qquad + \mathbb{P}[\text{Event}~R2 \big| \text{Event}~\bar{A}]  ~  \mathbb{E}[ I_{\mathcal{G}_t \backslash \mathcal{G}} \big| \text{Event}~\bar{A}, \text{Event}~R2 ] \nonumber\\
			& \overset{(b)}{\leq}  \mathbb{E}[ I_{\mathcal{G}_t \backslash \mathcal{G}} \big| \text{Event}~\bar{A} , \text{Event}~R1] + \left(\frac{\delta}{1 - \delta}\right) ~  \mathbb{E}[ I_{\mathcal{G}_t \backslash \mathcal{G}} \big| \text{Event}~\bar{A}, \text{Event}~R2 ] .
			\label{eq: IG}
		\end{align}
		where $(a)$ follows from the fact that Event $R2$ is the complement of Event $R1$ and the application of the law of total expectation. $(b)$ follows form the following:
		\begin{align*}
			\mathbb{P}[\text{Event}~R2 \big| \text{Event}~\bar{A}] \overset{(c)}{=} \frac{ \mathbb{P}[\text{Event}~R2 \cap \text{Event}~\bar{A}] }{ \mathbb{P}[\text{Event}~\bar{A}] }\overset{(d)}{\leq}  \frac{ \mathbb{P}[\text{Event}~R2] }{ \mathbb{P}[\text{Event}~\bar{A}] } \overset{(e)}{\leq} \frac{\delta}{1 - \delta}.
		\end{align*}
		where $(c)$ follows since $\mathbb{P}[\text{Event}~\bar{A}]  \neq 0$. $(d)$ follows since $[\text{Event}~R2 \cap \text{Event}~\bar{A} ] \subset \text{Event}~R2 $ and finally, $(e)$ follow from \eqref{eq: R2} and \eqref{eq: Abar}. 
		
		Now let us consider the first term of \eqref{eq: IG}, we have:
		\begin{align}
			& \mathbb{E}[ I_{\mathcal{G}_t \backslash \mathcal{G}} \big| \text{Event}~\bar{A} , \text{Event}~R1] =
			\mathbb{E} \Bigg[ \bigg\| \sum_{k \in \mathcal{G}_t \backslash \mathcal{G}} \left( \mu_t^{(k)} - \nabla f(\tilde{x}_{t-1}) \right) \bigg\|^2 \Bigg]    ~\text{under Event}~\bar{A} , \text{Event}~R1 \nonumber\\
			& \qquad \overset{(a)}{\leq} \alpha K~  \mathbb{E} \left[ \sum_{k \in \mathcal{G}_t \backslash \mathcal{G}}  \big\|   \mu_t^{(k)} - \nabla f(\tilde{x}_{t-1})   \big\|^2 \right]~\text{under Event}~\bar{A} , \text{Event}~R1 \nonumber\\
			& \qquad \overset{(b)}{\leq}  2 \alpha K~ \mathbb{E} \left[ \sum_{k \in \mathcal{G}_t \backslash \mathcal{G}}   \big\|   \mu_t^{(k)} - \mu_t^{\text{med}}   \big\|^2 \right] \nonumber \\
			& \qquad\qquad\qquad\qquad \qquad + 2 \alpha K~ \mathbb{E} \Bigg[ \underbrace{ \sum_{k \in \mathcal{G}_t \backslash \mathcal{G}}   \big\|\mu_t^{\text{med}}  -  \nabla f(\tilde{x}_{t-1})  \big\|^2}_{J_{\mathcal{G}_t \backslash \mathcal{G}}}   \Bigg] ~\text{under Event}~\bar{A} , \text{Event}~R1 \nonumber\\
			& \qquad \overset{(c)}{\leq} 2 \alpha K \left[ \alpha K    \frac{ 16 \mathcal{V}^2 C}{B_t}     \right]  + 2 \alpha K \left[ \alpha K     \frac{ 18 \mathcal{V}^2 C}{B_t}   \right]\nonumber \\
			& \qquad = \frac{68 \alpha^2 K^2 \mathcal{V}^2 C}{B_t} .
			\label{eq: IG1}
		\end{align}
		where $(a)$ follows from Lemma \ref{lem: Norm_Ineq} and the fact that $|\mathcal{G}_t \backslash \mathcal{G}| \leq \alpha K$, $(b)$ follows by adding and subtracting $\mu_t^{\text{med}}$ and applying Lemma \ref{lem: Norm_Ineq} and $(c)$ follows from:
		\begin{align*}
			& \mathbb{E}[ J_{\mathcal{G}_t \backslash \mathcal{G}} \big| \text{Event}~\bar{A}, \text{Event}~R1] \\ 
			& \qquad = \mathbb{P}[\text{Event}~A \big| \text{Event}~\bar{A} ,\text{Event}~ R1]   ~ \mathbb{E}[ J_{\mathcal{G}_t \backslash \mathcal{G}} \big|\text{Event}~ A, \text{Event}~\bar{A},\text{Event}~ R1 ] \\
			&  \qquad  \qquad \qquad   +  \mathbb{P}[\text{Event}~ A! \big|\text{Event}~ \bar{A} , \text{Event}~ R1]  ~ \mathbb{E}[ J_{\mathcal{G}_t \backslash \mathcal{G}} \big|\text{Event}~ A!, \text{Event}~ \bar{A},\text{Event}~ R1 ] \\
			& \qquad   \overset{(d)}{\leq}    \mathbb{E}[ J_{\mathcal{G}_t \backslash \mathcal{G}} \big|\text{Event}~ A, \text{Event}~ \bar{A},\text{Event}~ R1 ]  +  \bigg( \frac{\delta}{1 - 2 \delta} \bigg) \mathbb{E}[ J_{\mathcal{G}_t \backslash \mathcal{G}} \big|\text{Event}~ A!, \text{Event}~\bar{A},\text{Event}~ R1 ] \\
			& \qquad   \overset{(e)}{\leq}  \alpha K \frac{9 \mathcal{V}^2 C}{B_t} + \bigg( \frac{\delta}{1 - 2 \delta} \bigg) \alpha K 9 \mathcal{V}^2 \\
			& \qquad    \overset{(f)}{\leq} \alpha K \frac{18 \mathcal{V}^2 C}{B_t},
		\end{align*}
		where $(d)$ follows from:
		\begin{align*}
			\mathbb{P}[\text{Event}~A! \big|\text{Event}~ \bar{A} , \text{Event}~R1] & =  \frac{ \mathbb{P}[\text{Event}~A! , \text{Event}~\bar{A} , \text{Event}~R1]}{ \mathbb{P}[  \text{Event}~\bar{A} ,\text{Event}~ R1]} \\
			& \leq \frac{ \mathbb{P}[\text{Event}~A!]}{ \mathbb{P}[ \text{Event}~ \bar{A} , \text{Event}~R1]} \\
			& \overset{(g)}{\leq} \frac{\delta}{1 - 2 \delta},
		\end{align*}
		where $(g)$ follows from the fact that we have:
		$$\mathbb{P}[\text{Event}~  \bar{A} ~\cap~ \text{Event}~ R1] = \mathbb{P}[\text{Event}~ \bar{A} ] + \mathbb{P}[ \text{Event}~ R1] - \mathbb{P}[  \text{Event}~\bar{A}~ \cup~ \text{Event}~R1] \geq 1 - \delta + 1 - \delta - 1 = 1 - 2 \delta.$$
		Moreover, $(e)$ follows from the application of Lemma \ref{lem: T_mu} and Lemma \ref{lem: mu}. Finally, $(f)$ follows from choosing $\delta$ such that:
		\begin{align*}
			\mathrm{e}^{\frac{\delta B_t}{2(1 - 2 \delta)}} \leq \frac{2K}{\delta}.
		\end{align*}

		Now let us consider the second term of \eqref{eq: IG}, we have:
		\begin{align}
			& \mathbb{E}[ I_{\mathcal{G}_t \backslash \mathcal{G}} \big| \text{Event}~\bar{A}, \text{Event}~R2 ] = \mathbb{E} \Bigg[ \bigg\| \sum_{k \in \mathcal{G}_t \backslash \mathcal{G}} \left( \mu_t^{(k)} - \nabla f(\tilde{x}_{t-1}) \right) \bigg\|^2 \Bigg]    ~\text{under Event}~\bar{A} , \text{Event}~R2 \nonumber \\
			& \qquad \overset{(a)}{\leq} \alpha K~  \mathbb{E} \left[ \sum_{k \in \mathcal{G}_t \backslash \mathcal{G}}  \big\|   \mu_t^{(k)} - \nabla f(\tilde{x}_{t-1})   \big\|^2 \right]~\text{under Event}~\bar{A} , \text{Event}~R2 \nonumber\\
			& \qquad \overset{(b)}{\leq}  2 \alpha K~ \mathbb{E} \left[ \sum_{k \in \mathcal{G}_t \backslash \mathcal{G}}   \big\|   \mu_t^{(k)} - \mu_t^{\text{med}}   \big\|^2 \right] \nonumber \\
			& \qquad \qquad \qquad\qquad  + 2 \alpha K~ \mathbb{E} \left[ \sum_{k \in \mathcal{G}_t \backslash \mathcal{G}}   \big\|\mu_t^{\text{med}}  -  \nabla f(\tilde{x}_{t-1})  \big\|^2   \right] ~\text{under Event}~\bar{A} , \text{Event}~R2 \nonumber\\
			& \qquad \overset{(c)}{\leq} 2 \alpha K \left[ 16 \alpha K     \mathcal{V}^2    \right]  + 2 \alpha K \left[ 9  \alpha K  \mathcal{V}^2   \right]\nonumber \\
			& \qquad =  50 \alpha^2 K^2 \mathcal{V}^2 < 68 \alpha^2 K^2 \mathcal{V}^2 . 
			\label{eq: IG2}
		\end{align}
		where $(a)$ follows from Lemma \ref{lem: Norm_Ineq} and the fact that $|\mathcal{G}_t \backslash \mathcal{G}| \leq \alpha K$, $(b)$ follows by adding and subtracting $\mu_t^{\text{med}}$ and applying Lemma \ref{lem: Norm_Ineq} and $(c)$ follows from the fact that under Event $R2$ all nodes $k \in \mathcal{G}_t \backslash \mathcal{G}$ must satisfy Lemma \ref{lem: mu} statement (b). Now replacing \eqref{eq: IG1} and \eqref{eq: IG2} in \eqref{eq: IG} we get: 
		\begin{align}
			\label{eq: IG_Bound}
			\mathbb{E}[ I_{\mathcal{G}_t \backslash \mathcal{G}} \big| \text{Event}~\bar{A} ] \leq  \frac{68 \alpha^2 K^2 \mathcal{V}^2 C}{B_t}  + \bigg(  \frac{\delta}{1 - \delta} \bigg) 68 \alpha^2 K^2 \mathcal{V}^2 \overset{(a)}{\leq}    \frac{136 \alpha^2 K^2 \mathcal{V}^2 C}{B_t}.
		\end{align}
		where $(a)$ follows from choosing $\delta$ such that:
		\begin{align*}
			\mathrm{e}^{\frac{\delta B_t}{2(1 - \delta)}} \leq \frac{2K}{\delta}.
		\end{align*}
		Replacing \eqref{eq: norm_e_t_Abar_Good} and \eqref{eq: IG_Bound} in \eqref{eq: norm_e_t_Abar}, we have the bound on the first term of \eqref{eq: totalexpectation_et} as:
		\begin{align}
			\mathbb{E}\left[\| e_t \|^2 | \text{Event}~\bar{A}\right]  & \leq \frac{2}{(1 - \alpha)^2 K^2}  \Bigg(  \frac{K \mathcal{V}^2}{B_t}  +   \frac{136 \alpha^2 K^2 \mathcal{V}^2 C}{B_t} \Bigg) \nonumber\\
			& = \frac{2 \mathcal{V}^2}{(1 - \alpha)^2 K B_t} + \frac{272 \alpha^2  \mathcal{V}^2 C}{(1 - \alpha)^2 B_t}.
			\label{eq: norm_e_t_Abar_Bound}
		\end{align}
		We have bounded the first term of \eqref{eq: totalexpectation_et}. Now let us consider the second term of \eqref{eq: totalexpectation_et}:
		\begin{align}
			& \delta~ \mathbb{E}\left[\| e_t \|^2 | \text{Event}~\bar{A}!\right]  = \delta~  \mathbb{E}  \Bigg[ \bigg\| \frac{1}{|\mathcal{G}_t|} \sum_{k \in \mathcal{G}_t} \left( \mu_t^{(k)} - \nabla f(\tilde{x}_{t-1}) \right) \bigg\|^2   \Bigg] ~ \text{under Event}~\bar{A}! \nonumber \\
			& \quad \overset{(a)}{\leq} \frac{\delta}{(1 - \alpha)^2 K^2 }~  \mathbb{E}  \Bigg[ \big\|  \sum_{k \in \mathcal{G}_t} \left( \mu_t^{(k)} - \nabla f(\tilde{x}_{t-1}) \right) \big\|^2  \Bigg] ~ \text{under Event}~\bar{A}!\nonumber \\
			& \quad  \overset{(b)}{\leq} \frac{\delta}{(1 - \alpha)^2 K }~  \mathbb{E}  \Bigg[ \sum_{k \in \mathcal{G}_t}  \big\|  \mu_t^{(k)} - \nabla f(\tilde{x}_{t-1}) \big\|^2  \Bigg] ~ \text{under Event}~\bar{A}! \nonumber\\
			& \quad \overset{(c)}{\leq} \frac{2 \delta}{(1 - \alpha)^2 K } \Bigg(  \mathbb{E}  \bigg[ \sum_{k \in \mathcal{G}_t}  \big\|   \mu_t^{(k)} -   \mu_t^{\text{med}} \big\|^2  \bigg] +  \mathbb{E}  \bigg[ \sum_{k \in \mathcal{G}_t}  \big\|   \mu_t^{\text{med}} - \nabla f(\tilde{x}_{t-1})  \big\|^2  \bigg] \Bigg)~ \text{under Event}~\bar{A}! \nonumber \\
			& \quad \overset{(d)}{\leq} \frac{2 \delta}{(1 - \alpha)^2 K } \big(   16 K \mathcal{V}^2    +   9 K \mathcal{V}^2  \big)  = \delta \frac{50 \mathcal{V}^2 }{(1 - \alpha)^2} .
			\label{eq: norm_e_t_Abar!}
		\end{align}
		where $(a)$ follows from the fact that $|\mathcal{G}_t| \geq (1 - \alpha) K$. $(b)$ follows from Lemma \ref{lem: Norm_Ineq} and the fact that $|\mathcal{G}_t| \leq K$. $(c)$ follows from adding and subtracting $\mu_t^{\text{med}}$ and applying Lemma \ref{lem: Norm_Ineq}. Finally, $(d)$ follows by assuimg $\frac{2K}{\delta} \leq \mathrm{e}^{\frac{B_t}{2}}$ and $|\mathcal{G}_t| \leq K$. The assumption $\frac{2K}{\delta} \leq \mathrm{e}^{\frac{B_t}{2}}$ ensures that in the worst case the terms in inequality $(c)$ are bounded.

		Finally, replacing \eqref{eq: norm_e_t_Abar_Bound} and \eqref{eq: norm_e_t_Abar!} in \eqref{eq: totalexpectation_et} we get:
		\begin{align*}
			\mathbb{E} \| e_t \|^2 & \leq    \frac{2 \mathcal{V}^2}{(1 - \alpha)^2 K B_t} + \frac{272 \alpha^2  \mathcal{V}^2 C}{(1 - \alpha)^2 B_t} +  \delta \frac{50 \mathcal{V}^2 }{(1 - \alpha)^2} \\
			& \overset{(a)}{\leq}  \frac{4 \mathcal{V}^2}{(1 - \alpha)^2 K B_t} + \frac{272 \alpha^2  \mathcal{V}^2 C}{(1 - \alpha)^2 B_t} .
		\end{align*}
		where $(a)$ follows by choosing $\delta \leq \frac{1}{25 K B_t}$. Therefore, we have the bound. 
	\end{proof}
	
	\begin{lemma}
		\label{lem: T_mu}
		For any $t \in [T]$ and for all $k \in \mathcal{G}$ with probability at least $1 - \delta$ (we call this event as Event A) we have:\\
		(a): $\|\mu_t^{(k)} - \nabla f(\tilde{x}_{t-1})\| \leq \mathcal{V} \sqrt{ \frac{C}{B_t}}$.\\
		(b): This further implies that we have $\|\mu_t^{(k)} - \mu_t^{\text{med}}\| \leq 4 \mathcal{V} \sqrt{\frac{C}{B_t}}$ and $\|\mu_t^{\text{med}} - \nabla f(\tilde{x}_{t-1}) \| \leq 3 \mathcal{V} \sqrt{ \frac{C}{B_t}}$.\\
		where $C$ is defined as: $C = 2 \log\left(\frac{2 K}{\delta}\right)$. \vspace{0.08 in}\\
		Let us denote this event as Event $A$.
	\end{lemma}
	\begin{proof}
		{\em (a)} The proof follows by considering the random variable, $\nabla f(\tilde{x}_{t-1}; \xi_{t,i}^{(k)}) - \nabla f(\tilde{x}_{t-1})$ and from Assumption \ref{Ass: BoundedGradVar} we have $\|\nabla f(\tilde{x}_{t-1}; \xi_{t,i}^{(k)}) - \nabla f(\tilde{x}_{t-1})\| \leq \mathcal{V}$. Now applying Lemma \ref{lem: Pinelis} on the summation with $\|\mu_t^{(k)} - \nabla f(\tilde{x}_{t-1})\| = \| \frac{1}{B_t} \sum_{i = 1}^{B_t} \left(  \nabla f(\tilde{x}_{t-1}; \xi_{t,i}^{(k)}) - \nabla f(\tilde{x}_{t-1}) \right)\|$, we get the result. 
		
		{\em (b)} follows from the straightforward application of the above result. 
	\end{proof}
	
	Note that we call the above event by Event A and we have that the probability of Event A being true as: $\mathbb{P}[\text{Event A}] \geq 1 - \delta$ from Lemma \ref{lem: T_mu} and the discussion above. Now let us consider the case when Event A! (Complement of Event A) is true. In that case, the discussion above implies that we have $\mathbb{P}[\text{Event A!}] \leq  \delta$. 
	
	For the case when the set $|\mathcal{G}_t| < (1 - \alpha)K$ (please see Algorithm \ref{alg1}), we will make use of the following lemma.  
	
	\begin{lemma}
		\label{lem: mu}
		For any $t \in [T]$ and for all $k \in \mathcal{G}$ we have:\\
		(a): $\|\mu_t^{(k)} - \nabla f(\tilde{x}_{t-1})\| \leq \mathcal{V}$.\\
		(b): This further implies that we have $\|\mu_t^{(k)} - \mu_t^{\text{med}}\| \leq 4 \mathcal{V}$ and $\|\mu_t^{\text{med}} - \nabla f(\tilde{x}_{t-1}) \| \leq 3 \mathcal{V}$.
	\end{lemma}
	\begin{proof}
		{\em (a)} The proof follows from the definition of $\mu_t^{(k)}$ and the application of the triangle inequality along with Assumption \ref{Ass: BoundedGradVar}. 
		
		{\em (b)} follows from the straightforward application of the above. 
	\end{proof}

	\begin{lemma}
		\label{lem: GeomRV}
		If $N \sim \text{Geom}(\Gamma)$ for $\Gamma > 0$. Then for any sequence $D_0, D_1, \ldots$ with $\mathbb{E}|D_N| < \infty$, we have
		$$ \mathbb{E}[D_N - D_{N+1}]  =  \left( \frac{1}{\Gamma} - 1 \right) (D_0 - \mathbb{E}D_N)$$
	\end{lemma}
	\begin{proof}
		Proof follows from \citet{Lei_Jordan_SCSG}.
	\end{proof}
	
	\begin{lemma}
		\label{lem: Finite}
		For step size $\eta_t \leq \frac{1}{3 L B_t^{{2}/{3}} }$, we have:
		\begin{enumerate}[(i)]
			\item $\mathbb{E}\| \tilde{x}_t - \tilde{x}_{t-1} \|^2 < \infty.$
			\item $\mathbb{E}(f(\tilde{x}_t) - f(\tilde{x}^\ast)) < \infty.$
			\item $\mathbb{E} \| \nabla f(\tilde{x}_t) \|^2 < \infty.$
			\item $\mathbb{E}| \langle e_t, \tilde{x}_t - \tilde{x}_{t-1}  \rangle | < \infty$.
			\item $\mathbb{E} | \langle  e_t , \nabla f (\tilde{x}_t) \rangle | < \infty .$
		\end{enumerate}
	\end{lemma}
	\begin{proof}
		The lemma is proven using induction and follows the same structure as the proof in \citet{Lei_Jordan_SCSG}. The second inequality \eqref{eq: For_Finite_T1} in the proof of theorem yields
		\begin{align}
			\mathbb{E}_{\xi_n^t}  f(x_{n + 1}^t)  &  \leq  f(x_n^t) - \eta_t (1 - L\eta_t) \| \nabla f(x_n^t) \|^2  \nonumber\\
			&  \qquad \qquad   -  \eta_t      \langle  e_t  ,  \nabla f(x_n^t)   \rangle + \frac{L^3 \eta_t^2}{2} \| x_n^t - x_0^t \|^2 + L \eta_t^2 \| e_t \|^2,
			\label{eq: FinitenessPf_1}
		\end{align}
		using Young's inequality  $  \langle  a , b \rangle \leq   \frac{1}{2 \beta}  \|a \|^2 +  \frac{\beta}{2}   \| b\|^2   $ for any $\beta > 0$, on $- \eta_t      \langle  e_t  ,  \nabla f(x_n^t)   \rangle$ with $\beta = \frac{1}{2}$ we get:
		\begin{align*}
			- \eta_t      \langle  e_t  ,  \nabla f(x_n^t)   \rangle  \leq   \eta_t \| e_t\|^2 + \frac{\eta_t}{4} \|  \nabla f(x_n^t) \|^2 .
		\end{align*}
		Moreover, using the fact that $\eta_t \leq \frac{1}{3 L B_t^{{2}/{3}} } \leq \frac{1}{4L}$ since $B_t \geq 16$ and rearranging the terms in \eqref{eq: FinitenessPf_1} we have
		\begin{align}
			& \eta_t \bigg(1 - L\eta_t -  \frac{1}{4} \bigg) \|  \nabla f(x_n^t) \|^2 	    \nonumber\\
			& \qquad \qquad \qquad \leq  f(x_n^t) - \mathbb{E}_{\xi_n^t}  f(x_{n + 1}^t)  + \frac{L^3 \eta_t^2}{2} \| x_n^t - x_0^t \|^2 + \eta_t (1 + L \eta_t) \| e_t \|^2 \nonumber \\
			&	\eta_t \bigg(1 - \frac{1}{4} -  \frac{1}{4} \bigg) \|  \nabla f(x_n^t) \|^2 	  \nonumber\\
			& \qquad \qquad \qquad \leq  f(x_n^t) - \mathbb{E}_{\xi_n^t}  f(x_{n + 1}^t)   + \frac{L^3 \eta_t^2}{2} \| x_n^t - x_0^t \|^2 + \eta_t \bigg(1 + \frac{1}{4} \bigg) \| e_t \|^2 \nonumber\\
			&	\eta_t   \|  \nabla f(x_n^t) \|^2 	  \leq 2 \big( f(x_n^t) - \mathbb{E}_{\xi_n^t}  f(x_{n + 1}^t) \big)  +  L^3 \eta_t^2  \| x_n^t - x_0^t \|^2 + \frac{5 \eta_t}{2}  \| e_t \|^2 .
			\label{eq: FinitenessPf_2}
		\end{align}
		Now using the first inequality \eqref{eq: For_Finite_L5} in Proof of Lemma \ref{lem: Norm_Sq}, we have
		\begin{align*}
			\mathbb{E}_{\xi_n^t} \| x_{n + 1}^t - x_0^t \|^2 &  \leq   (1 + \eta_t^2 L^2)  \| x_n^t  - x_0^t  \|^2  -  2\eta_t \langle  \nabla f(x_n^t)   , x_n^t - x_0^t  \rangle \nonumber\\
			& \qquad \qquad \qquad  - 2 \eta_t \langle   e_t , x_n^t - x_0^t   \rangle + 2 \eta_t^2 \| \nabla f(x_n^t) \|^2  + 2 \eta_t^2 \|e_t \|^2 ,
		\end{align*}
		using Young's inequality  $  \langle  a , b \rangle \leq   \frac{\beta}{2 }  \|a \|^2 +  \frac{1}{2\beta}   \| b\|^2   $ for any $\beta > 0$, on $-  2\eta_t \langle  \nabla f(x_n^t)   , x_n^t - x_0^t  \rangle $ and $- 2 \eta_t \langle   e_t , x_n^t - x_0^t   \rangle$ with $\beta = 8 \eta_t B_t$ we get:
		\begin{align*}
			-  2\eta_t \langle  \nabla f(x_n^t)   , x_n^t - x_0^t  \rangle & \leq  8 \eta_t^2 B_t \| \nabla f(x_n^t) \|^2 + \frac{1}{8 B_t} \|x_n^t - x_0^t \|^2 \\
			- 2 \eta_t \langle   e_t , x_n^t - x_0^t   \rangle & \leq  8 \eta_t^2 B_t  \|e_t\|^2   + \frac{1}{8 B_t}  \| x_n^t - x_0^t  \|^2,
		\end{align*}
		Therefore, we get:
		\begin{align}
			&	\mathbb{E}_{\xi_n^t} \| x_{n + 1}^t - x_0^t \|^2 \nonumber\\
			& \quad   \leq   \bigg(1 + \eta_t^2 L^2  + \frac{1}{4 B_t} \bigg)  \| x_n^t  - x_0^t  \|^2  + (2 \eta_t^2 + 8 \eta_t^2 B_t)   \|  \nabla f(x_n^t) \|^2   + (  2 \eta_t^2 + 8 \eta_t^2 B_t ) \|   e_t  \|^2  \nonumber \\
			& \quad   \overset{(a)}{\leq}   \bigg(1 +   \frac{13}{36 B_t} \bigg)  \| x_n^t  - x_0^t  \|^2  + 10 \eta_t^2 B_t   \|  \nabla f(x_n^t) \|^2   + 10 \eta_t^2 B_t  \|   e_t  \|^2   
			\label{eq: FinitenessPf_3}
		\end{align}
		where $(a)$ used the fact that we $\eta_t L \leq \frac{1}{3 B_t^{{2}/{3}} }$. Now plugging \eqref{eq: FinitenessPf_2} into \eqref{eq: FinitenessPf_3} we get:
		\begin{align}
			&	\mathbb{E}_{\xi_n^t} \| x_{n + 1}^t - x_0^t \|^2  \leq  \bigg(1 +   \frac{13}{36 B_t} \bigg)  \| x_n^t  - x_0^t  \|^2  + 10 \eta_t^2 B_t  \|   e_t  \|^2   \nonumber\\
			& \qquad \qquad	+ 20 \eta_t B_t \big(  f(x_n^t) - \mathbb{E}_{\xi_n^t} f(x_{n + 1}^t) \big) + 10 \eta_t^3 L^3 B_t \| x_n^t - x_0^t \|^2 + 25 \eta_t^2 B_t \| e_t \|^2  \nonumber \\
			&  \leq \bigg(1 +   \frac{13}{36 B_t}  + 10 \eta_t^3 L^3 B_t \bigg)  \| x_n^t  - x_0^t  \|^2   + 20 \eta_t B_t \big(  f(x_n^t) - \mathbb{E}_{\xi_n^t} f(x_{n + 1}^t) \big) +  35 \eta_t^2 B_t \| e_t \|^2 \nonumber \\
			& \overset{(a)}{\leq} \bigg(1 +   \frac{711}{972 B_t}  \bigg)  \| x_n^t  - x_0^t  \|^2   + 20 \eta_t B_t \big(  f(x_n^t) - \mathbb{E}_{\xi_n^t} f(x_{n + 1}^t) \big) +  35 \eta_t^2 B_t \| e_t \|^2 
			\label{eq: FinitenessPf_4}
		\end{align}
		where $(a)$ follows from using $\eta_t L \leq \frac{1}{3 B_t^{{2}/{3}} }$. Let us assume
		$$ L_n^t  = 20 \eta_t B_t \mathbb{E} \big(  f(x_n^t) - f(\tilde{x}^\ast) \big)  + \mathbb{E} \| x_n^t - x_0^t \|^2$$
		Taking expectation over \eqref{eq: FinitenessPf_4} we get:
		\begin{align*}
			L_{n + 1}^t \leq \bigg(  1 + \frac{711}{972 B_t} \bigg) L_n^t   +  35 \eta_t^2 B_t \mathbb{E} \| e_t \|^2
		\end{align*} 
		Denoting $\gamma =  \frac{711}{972} < 1$ we have: 
		\begin{align*}
			L_{n + 1}^t & \leq \bigg(  1 + \frac{\gamma}{ B_t} \bigg) L_n^t   +  35 \eta_t^2 B_t \mathbb{E} \| e_t \|^2 \\
			L_{n + 1}^t   +   \frac{35 \eta_t^2 B_t^2 \mathbb{E} \| e_t \|^2 }{\gamma}	&  \overset{(a)}{\leq}  \bigg(  1 + \frac{\gamma}{ B_t} \bigg) L_n^t   + \bigg(  1 + \frac{\gamma}{ B_t} \bigg)   \frac{35 \eta_t^2 B_t^2 \mathbb{E} \| e_t \|^2 }{\gamma} \\
			L_{n + 1}^t   +   \frac{35 \eta_t^2 B_t^2 \mathbb{E} \| e_t \|^2 }{\gamma}	&   \leq  \bigg(  1 + \frac{\gamma}{ B_t} \bigg) \bigg( L_n^t   +     \frac{35 \eta_t^2 B_t^2 \mathbb{E} \| e_t \|^2 }{\gamma} \bigg)
		\end{align*} 
		where $(a)$ follows from adding and subtracting $ \frac{35 \eta_t^2 B_t^2 \mathbb{E} \| e_t \|^2 }{\gamma}$ on both sides. 
		this implies that we have:
		\begin{align*}
			L_{n}^t \leq \bigg(  1 + \frac{\gamma}{B_t} \bigg)^n  \bigg( L_0^t   +  \frac{35 \eta_t^2 B_t^2 \mathbb{E} \| e_t \|^2}{\gamma} \bigg)
		\end{align*}
		Since we have: $N_t \sim \text{Geom}\left(\frac{B_t}{B_t + 1}\right)$, and assuming $N_t$ can be $0$ we have
		$$ \mathbb{P}[N_t = n] = \frac{1}{B_t + 1} \left(  \frac{B_t}{B_t + 1}  \right)^n \leq \left(  \frac{B_t}{B_t + 1}  \right)^n .$$
		Now the term: 
		$$\mathbb{E} \bigg[ \bigg(  1 + \frac{\gamma}{  B_t} \bigg)^{N_t} \bigg] \leq \sum_{ n \geq 0} \left(  \frac{ B_t + \gamma}{B_t}   \times \frac{B_t}{B_t + 1}  \right)^n  =  \sum_{ n \geq 0} \left(  \frac{ B_t + \gamma}{  B_t +  1}  \right)^n  \overset{(a)}{=} \frac{ B_t + 1}{1 - \gamma}.$$
		$(a)$ follows since $\gamma = \frac{711}{972} < 1$.	This implies that:
		$$ \mathbb{E} L_{N_t}^t \leq \frac{ B_t + 1}{1 - \gamma}\big( L_0^t   +  70 \eta_t^2 B_t \mathbb{E} \| e_t \|^2 \big) .$$
		which is finite since $\mathbb{E}\| e_t \| < \infty$ is finite by Lemma \ref{lem: e_t} as well as the filtering rule of Algorithm \ref{alg1}. The induction hypothesis implies that $\mathbb{E} L_{N_t} < \infty$.
		
		All the claims follow. 
	\end{proof}
	
	\begin{lemma}[\citet{Alistarh_NIPS_2018} Lemma 2.4]
		\label{lem: Pinelis}
		Let the sequence of random variables $X_1, X_2, \ldots, X_N \in \mathbb{R}^d$ represent a random process such that we have $\mathbb{E}[X_n | X_1, \ldots, X_{n-1}] = 0$ and $\| X_n \| \leq M$. Then, 
		$$ \mathbb{P}[\| X_1 + \ldots + X_N \|^2 \leq 2 \log(2/\delta) M^2 N] \geq 1 - \delta.$$
	\end{lemma}
	
	\begin{lemma}
		\label{lem: Norm_Ineq}
		For $X_1, X_2, \ldots, X_n \in \mathbb{R}^d$, we have 
		\begin{align*}
			\|X_1 + X_2 + \ldots + X_n \|^2 \leq n \| X_1\|^2 + n \| X_2\|^2+ \ldots + n \| X_n\|^2.
		\end{align*}
	\end{lemma}
	
\end{document}